\documentclass[]{article}
\usepackage{amsmath}
\usepackage{amssymb}
\usepackage{mathtools}
\usepackage{graphics}
\usepackage{amsfonts}
\usepackage{tikz} 
\usepackage{forloop}
\usepackage{multirow}
\usepackage{float}
\usepackage[a4paper]{geometry}
\usepackage[toc,title,page]{appendix}
\usepackage{amsthm}
\usepackage{parskip}
\usepackage{mathrsfs}
\usepackage{graphicx}
\usepackage{xtab}
\usepackage{tikz}
\usepackage{forloop}
\usepackage{multirow}
\usepackage{caption}
\usepackage{subcaption}
\usepackage{enumerate}
\usepackage{amsthm}
	\usepackage[utf8]{inputenc}
	\usepackage{setspace}
 
\usepackage[
backend=biber,
style=numeric,
sorting=none
]{biblatex}

\usepackage[section]{placeins}
\usepackage{enumitem}

\usepackage[capitalise]{cleveref}

\usetikzlibrary{automata,positioning,arrows}
\tikzset{node distance=1.5cm}
\title{Eco-evolutionary dynamics of trait-structured, sexually
reproducing interacting populations with time-dependent birth and mortality rates}
\author{}

\newtheorem{lemma}{Lemma}

\newtheorem{theorem}{Theorem}
\newtheorem{proposition}{Proposition}

\theoremstyle{remark}
\newtheorem{remark}{Remark}

\usepackage [english]{babel}
\usepackage [autostyle, english = american]{csquotes}
\MakeOuterQuote{"}

\author{
  Manh Hong Duong\\
  \texttt{h.duong@bham.ac.uk}
  \and
  Fabian Spill\\
  \texttt{f.spill@bham.ac.uk}
    \and
 Blaine Van Rensburg\\
\texttt{BXV114@student.bham.ac.uk}
}



\begin{document}

\maketitle

% REQUIRED
\begin{abstract}
The coupling between evolutionary and ecological changes (eco-evolutionary dynamics) has been shown to be relevant among diverse species, and is also of interest outside of ecology, i.e. in cancer evolution. These dynamics play an important role in determining survival in response to climate change, motivating the need for mathematical models to capture their complex interplay. Models incorporating eco-evolutionary dynamics often sacrifice analytical tractability to capture the complexity of real systems, do not explicitly consider the effect of population heterogeneity, or focus on long-term behaviour. In order to capture population heterogeneity, both transient, and long-term dynamics, while retaining tractability, we generalise a moment-based method applicable in the regime of small segregational variance to the case of time-dependent mortality and birth. These results are applied to a predator-prey model, where ecological parameters such as the contact rate between species are trait-structured. The trait-distribution of the prey species is shown to be approximately Gaussian with constant variance centered on the mean trait, which is asymptotically governed by an autonomous ODE. In this way, we make explicit the impact of eco-evolutionary dynamics on the transient behaviour and long-term fate of the prey species. 

\end{abstract}

\textbf{Keywords:} PDEs, integro-differential equation, infinitesimal model, moments, eco-evolutionary dynamics

\section{Introduction}
\subsection{Background on eco-evolutionary dynamics}
Predictive studies on the response of biodiversity to climate change are important to inform scientists and policy-makers of risks to biodiversity and possible interventions \cite{bellard2012impacts}. Prediction is complicated since the response of a given species to a changing environment depends on intraspecific variation (phenotype heterogeneity) \cite{moran2016intraspecific,mimura2017understanding,bolnick2011intraspecific}, as well as complex interactions between other species, and the environment more generally (see \cite{hendry2011evolutionary}). Complicating matters is the increasing evidence of coupling between evolutionary and ecological changes, eco-evolutionary dynamics (reviewed in \cite{fussmann2007eco,pelletier2009eco, govaert2019eco,post2009eco,hairston2005rapid}). The eco-evolutionary perspective is also valuable in informing our approach to cancer therapies \cite{adler2019cancer}. 

Mathematical models help illuminate the impacts of eco-evolutionary dynamics. Specifically, they have predicted that the long-term fate of predator and prey species are impacted by the rate at which evolutionary processes occur \cite{marrow1996evolutionary,dieckmann1996dynamical}, and that host-parasite dynamics are qualitatively altered when bi-directional feedback between ecology and evolution is specifically included \cite{ashby2019understanding,best2023fluctuating}, i.e. cyclic dynamics emerge due to evolution. However, these models do not explicitly include the processes that maintain variance on the population, such as mutation. The limiting solution as mutation rate approaches $0$ for selection-mutation models may be qualitatively different than the solution with zero mutation rate (see \cite{lorenzi2020asymptotic} for an example), so it is important to determine the behaviour of models for small, but non-zero, population heterogeneity

Integro-differential equations provide a mathematical framework to model natural selection in phenotypically heterogeneous populations, and have tremendous utility in understanding how a diverse range of factors affect evolution, such as time or space heterogeneity  \cite{iglesias2021selection,roques2020adaptation,figueroa2018long, bouin2015hamilton},  the competition between populations exhibiting different dispersal rates \cite{lam2023hamilton, bouin2014travelling}, and the impact of phenotypic heterogeneity on the response of cancers to therapy \cite{ardavseva2020comparative,lorenzi2015dissecting,lorenzi2016tracking}. However, these models do not incorporate eco-evolutionary dynamics. Integro-differential equation models which do incorporate eco-evolutionary dynamics have provided unique insights into evolutionary branching \cite{diekmann2005dynamics}, the interaction between stromal cells and stem cells \cite{nguyen2019adaptive}, and the co-evolution of predators and prey \cite{calsina1994non}. While these models  specifically involve mutational processes, they focus on stationary solutions (the evolutionary stable strategies in the language of adaptive dynamics) and so miss the transient dynamics which may be of interest in applications, especially if realistic parameters would require a longer time to converge than the time-frame of interest. Moreover, this means they are not readily adapted to the case where the attractor of the system is not a fixed point. 

The aim of this paper is therefore to study the full eco-evolutionary dynamics of trait-structured, sexually reproducing interacting populations, in which the death and/or birth rates are time-dependent. We employ a generalised moment based method, which enables us to study eco-evolutionary dynamics including the changes in population heterogeneity while retaining tractability.

\subsection{Models}

\subsection*{A single-species model} We first consider the following nonlinear integro-differential evolution equation that models the evolutionary dynamics of a trait-structured, sexually reproducing population:
\begin{equation} \label{eqn:main}
\begin{cases}
\varepsilon^2 \partial_t   n_\varepsilon(x,t)= r_{\varepsilon}(t) T_{\varepsilon}  [ n_\varepsilon](x,t)  
-   \left( \tilde{m}_{\varepsilon}(x,t)  + \kappa \rho_\varepsilon(t) \right) n_\varepsilon(x,t) ,\\
\rho_\varepsilon(t)= \int_{\mathbb R}   n_\varepsilon(y,t)dy,\\
n_\varepsilon(x,0)=n_{\varepsilon,0}(x),
\end{cases}
\end{equation}
where $(x,t)\in{}\mathbb{R}\times(0,\infty)$ and
\begin{equation}
\label{eq: Teps}
T_{\varepsilon}[n](x) =\int_{\mathbb R} \int_{\mathbb R} \Gamma_\varepsilon \left(x-\frac{(y+y')}{2}\right) \frac{n(y)n(y')}{\int_{\mathbb{R}}n(z)dz} dydy',\quad \Gamma_\varepsilon(x)= \frac{1}{\varepsilon\sqrt{\pi}} \exp\left(-\frac{x^2}{\varepsilon^2}\right).
\end{equation}
The unknown function $n_\varepsilon(x,t)$ describes the density of the population characterized by a phenotypic trait $x$ at time $t$; $r_\varepsilon: t\in(0,\infty)\mapsto r_\varepsilon(t)\in\mathbb{R}$ represents the time-dependent birth rate while $\tilde{m}_\varepsilon:(x,t)\in \mathbb{R}\times(0,\infty)\mapsto m_\varepsilon(x,t)\in \mathbb{R}$ captures a trait and time-dependent per capita mortality rate. We allow both of these demographic rates to depend on $\varepsilon$ because we aim to apply this to a system where these terms depend on the population $n_\varepsilon(x,t)$. 
 In addition, the
population is also subject to a mortality rate due to uniform competition between individuals, represented by the term $\kappa \rho_\varepsilon$  where $\rho_\varepsilon$ is the total size of the population. The operator $T_\varepsilon$ given in \eqref{eq: Teps} describes the sexual reproduction and is often referred to as the “infinitesimal model” in the literature. The infinitesimal model, first introduced in \cite{fisher1919xv}, assumes that there are many loci, each with a small effect on the fitness, which contribute to the observed one-dimensional trait $x$, and that the traits of the offspring are normally distributed around the average trait of their parents $\frac{(y+y')}{2}$ with a variance $\varepsilon^2$. The infinitesimal model has been widely used in biology and population dynamics; we refer the reader to \cite{barton2017infinitesimal} for a rigorous derivation from a probabilistic viewpoint and to~\cite{turelli2017commentary} for a detailed account of the infinitesimal model. The assumption of small segregational variance $\varepsilon^2$ corresponds to the well-established weak selection regime \cite{burger2000mathematical}, and the coefficient $\varepsilon^2$ multiplying the time derivative is the appropriate time-scaling to capture evolutionary processes. The model \eqref{eqn:main} with time-independent per capita mortality and birth rates, that is, $\tilde{m}_\varepsilon=\tilde{m}_\varepsilon(x)$ and $r_\varepsilon=\mathrm{const}$, is studied recently in \cite{GHM}. In this paper, we extend the results of \cite{GHM} to the case of time-dependent per capita mortality and birth rates.   
% role of epsilon, from :the weak selection that is often used in population genetics (Wakeley 2005), and which
%has been used in many mathematical studies involving the infinitesimal model (Patout 2023; Raoul 2017).
\subsection*{A predator-prey model with trait-structured prey}
Next, we consider the following predator-prey system 
\begin{equation} \label{eqn:evolving_prey}
\begin{cases}
\varepsilon^2 \partial_t   n_{1,\varepsilon}= r_1T_{{1,\varepsilon}}  [ n_{1,\varepsilon}]  
-   \left(\frac{f(x)\rho_{2,\varepsilon}}{1+h\bar{f}[q_{1,\varepsilon}](t)\rho_{1,\varepsilon}}  + (\kappa_1-\delta(x)) \rho_{1,\varepsilon} \right) n_{1,\varepsilon} ,\\
\tau \partial_t   \rho_{2,\varepsilon}= \left(\gamma{}\frac{\bar{f}[q_{1,\varepsilon}](t)\rho_{1,\varepsilon}}{1+h\bar{f}[q_{1,\varepsilon}](t)\rho_{1,\varepsilon}} -\kappa_2\rho_{2,\varepsilon}\right)\rho_{2,\varepsilon},\\
\rho_{1,\varepsilon}(t)= \int_{\mathbb R}   n_{1,\varepsilon}(y,t)dy,~~q_{1,\varepsilon}=n_{1,\varepsilon}/\rho_{1,\varepsilon},~~\bar{f}[q_{1,\varepsilon}](t)=\int_{}q_{1,\varepsilon}(y,t)f(y)dy,\\
n_{1,\varepsilon}(x,0)=n_{{1,\varepsilon},0}(x), ~~\rho_{2,\varepsilon}(0)=\rho_{2,\varepsilon,0}.
\end{cases}
\end{equation}

This model extends the single-species model \eqref{eqn:main} to the case of predator-prey interacting populations. Here $n_{1,\varepsilon}:(x,t)\in\mathbb{R}\times (0,\infty)\mapsto \mathbb{R}$ and $\rho_{2,\varepsilon}: t\in (0,\infty)\mapsto \mathbb{R}$ respectively describe the density of the prey with a trait $x$ and the predator population size at time $t$. The trait $x$ is the `riskiness' of prey foraging strategy, i.e. foraging near predator habitats.
The term $\frac{f(x)\rho_{2,\varepsilon}}{1+h\bar{f}[q_{1,\varepsilon}](t)\rho_{1,\varepsilon}}$  represents the Holling Type II functional response of the predator, where contact rate is given by the function $f:x\in\mathbb{R}\mapsto{}f(x)\in [0,\infty)$. The handling time per prey is $h$, and we have assumed that the total handling time depends on the average contact rate.  The carrying capacity of preys is given by $(\kappa_1-\delta(x))$, where the function $\delta:x\in\mathbb{R}\mapsto \delta(x)\in\mathbb[0,\infty)$ accounts for the decrease in competition due to obtaining more resources. As such we require that $\delta(x)<\kappa_1$. Both $f(x)$ and $\delta(x)$ are assumed to be smooth, non-negative, and increasing. The parameter $\kappa_{2}$  corresponds to the predator competition strength,  $\gamma$ is the conversion factor of prey into predator, and $\tau$ sets the timescale for predator dynamics.

In this paper, we focus on the special regime where $\tau\ll\varepsilon^2$ for which \eqref{eqn:evolving_prey} reduces to a single evolution equation for $n_{1,\varepsilon}$, with a modified mortality rate
\begin{equation} \label{eqn:evolving_prey_simplified}
\begin{cases}
\varepsilon^2 \partial_t   n_{1,\varepsilon}= r_1T_{{1,\varepsilon}}  [ n_{1,\varepsilon}]  
-   \left( \frac{\gamma}{\kappa_2}\frac{f(x)\bar{f}[q_{1,\varepsilon}](t)\rho_{1,\varepsilon}}{(1+h\bar{f}[q_{1,\varepsilon}](t)\rho_{1,\varepsilon})^2}+(\kappa_1-\delta(x))\rho_{1,\varepsilon} \right) n_{1,\varepsilon} ,\\
\rho_{1,\varepsilon}(t)= \int_{\mathbb R}   n_{1,\varepsilon}(y,t)dy,~~q_{1,\varepsilon}=n_{1,\varepsilon}/\rho_{1,\varepsilon},~~\bar{f}[q_{1,\varepsilon}](t)=\int_{}q_{1,\varepsilon}(y,t)f(y)dy,\\
n_{1,\varepsilon}(x,0)=n_{{1,\varepsilon},0}(x). 
\end{cases}
\end{equation}
Compared to the single-species model \eqref{eqn:main}, the mortality rate in the reduced model \eqref{eqn:evolving_prey_simplified} depends on the unknown $n_{1,\varepsilon}$ itself, thus adding further nonlocality and nonlinearity. The reduction from the coupled system \eqref{eqn:evolving_prey} to the single equation \eqref{eqn:evolving_prey_simplified} under the 
assumption that $\tau\ll \varepsilon$  resembles the reduction from the parabolic-parabolic Patlak-Keller-Segel model to the parabolic-elliptic one, see for instance~ \cite{keller1970,hadeler2004langevin,blanchet2011}. We will investigate the full coupled system  \eqref{eqn:evolving_prey} in future works.

\subsection{Summary of the main results}

The main results of the paper can be summarised as follows (see Section \ref{sec: assumptions and main results} for precise assumptions and statements). Our first result, Theorem 1, shows that the trait-distribution for a population whose dynamics are given by (\ref{eqn:main}) are approximately Gaussian with constant variance $\varepsilon^2$, centred on a mean trait which is governed by an ODE, the canonical equation of adaptive dynamics. This result is obtained under the assumption that the initial trait distribution is sufficiently concentrated not too far from an optimal trait, and the optimal trait does not vary too quickly. Our second main result, Theorem 2, shows that for the reduced predator-prey model (\ref{eqn:evolving_prey_simplified}), where the mortality is now dependent on the population size and trait distribution, the same Gaussian approximation as before holds. In this case, the obtained evolution equation for the mean trait now expresses explicitly the influence of eco-evolutionary dynamics.

%Our plan is generalise the results of \cite{GHM} to the single species model (\ref{eqn:main}), under suitable hypotheses (given in the next section). We may then show these same hypotheses are satisfied for the interacting species model (\ref{eqn:evolving_prey_simplified}), given assumption on $f$, $\delta$, and the initial condition. Our approach requires careful control of the population size and mean trait. 

\subsection{Organization of the paper}
The rest of the paper is organized as follows. In Section \ref{sec: assumptions and main results} we provide precise assumptions and the main results of the paper together with an overview of the proofs. The proofs of Theorem \ref{thm:main} and Theorem \ref{thm:IP_main} are then given in Section \ref{sec: proof of 1st theorem} and Section \ref{sec: proof of 2nd thm} respectively. In Section \ref{sec: numerics} we show numerical simulations to demonstrate the analytical results. In the last Section \ref{sec:Conclusion} we provide further discussion and outlook.

\section{Assumptions and main results}
\label{sec: assumptions and main results}
In this section, we provide precise assumptions and the main results concerning the asymptotic analysis of the co-evolutionary dynamics \eqref{eqn:main} and \eqref{eqn:evolving_prey_simplified}.
\subsection{Single Species Model}
We first focus on (\ref{eqn:main}) and recall some useful facts and notations from \cite{GHM}. We define the trait distribution $q_\varepsilon$ as
\begin{equation}
\label{eq: qeps}
q_\varepsilon(x,t)=\frac{n_\varepsilon(x,t)}{\rho_\varepsilon(t)}=\frac{n_\varepsilon(x,t)}{\int_{\mathbb{R}}n_\varepsilon(y,t)\,dy}.
\end{equation}
Then it follows from \eqref{eqn:main} that $q_\varepsilon(x,t)$ solves the following integro-differential equation
\begin{equation} \label{eqn:distribution}
\varepsilon^2 \partial_t q_\varepsilon = r_\varepsilon(t) \left(\widetilde T_{\varepsilon}[q_\varepsilon] - q_\varepsilon \right) 
-  \left( \tilde{m}_{\varepsilon}(x,t)  - \int_{\mathbb R}  \tilde{m}_{\varepsilon}(x,t)q_\varepsilon(x,t) dx \right) q_\varepsilon ,
\end{equation}
with
\[
\widetilde T_{\varepsilon}[q_\varepsilon](x) =  \int_{\mathbb R} \int_{\mathbb R} \Gamma_\varepsilon \left(x-\frac{(y+y')}{2}\right) q_\varepsilon(y,t)q_\varepsilon(y',t) dydy'.
\]
We define the central moments of the trait distribution by
\begin{subequations}
\label{eq: central moments}
    \begin{align}
   &M_{\varepsilon,1}(t)=\int_{\mathbb{R}}xq_\varepsilon(x,t)dx,\\
    &M_{\varepsilon,k}^c(t)=\int_{\mathbb{R}}(x-M_{\varepsilon,1})^kq_\varepsilon(x,t)dx,~~    M_{\varepsilon,k}^{|c|}(t)=\int_{\mathbb{R}}|x-M_{\varepsilon,1}|^kq_\varepsilon(x,t)dx.
\end{align}
\end{subequations}
We now specify the assumptions of the paper, starting with those for the mortality rate.
\subsection*{Assumptions on the mortality rate}
We assume that $\tilde{m}_{\varepsilon}(x,t)$ can be written as $\tilde{m}_{\varepsilon}(x,t)={m}_{\varepsilon}(x-X_\varepsilon(t),t)$, for some function $X_\varepsilon(t)$ (which could be interpreted as a time-dependent environmental shifting of the minimum of mortality rate),  such that ${m}_\varepsilon(x,t)\in{}C^{\infty,1}(\mathbb{R}^2)$ satisfies the following properties for some $L >0$:
\ \\
\begin{itemize}
\item[(H0)] $m_\varepsilon(x,t)\geq 0$ for all $(x,t)\in \mathbb{R}\times (0,\infty)$ with $\inf_{x\in \mathbb{R}}m_\varepsilon(x,t) = 0$ for all $t\in (0,\infty)$. 
\item[(H1)] $\partial_{x}m_\varepsilon(0,t)  = 0$ and $A_0  \leq \partial_{xx}m_\varepsilon(x,t)$ for all $x\in (-L,L)$,  for some constant $A_{0}>0.$ 
\item[(H2)] $\max_{x \in [-L,L]} m_\varepsilon(x,t) < r_L $, and $r_\varepsilon(t)>r_L$, where $r_L>0$ is independent of $\varepsilon$.
\item[(H3)] $|\partial_{xxx}m_\varepsilon(x,t)| \leq A_m  (1+|x|^{p-1})$ for all $x\in \mathbb{R}$ and some exponent $p \in \mathbb N^*$ with $p>1$ and constant $A_m \in (0,\infty)$. 
\end{itemize}
\ \\
\begin{remark}
The biological interpretation of (H0) is that mortality is always positive. The assumption (H1) ensures that the local minimum mortality at $x=0$ is strict. The assumption (H2) ensures that the birth rate is positive  near that local minimum. The last assumption (H3) is technical to bound the growth at infinity.
\end{remark}
\subsection*{Assumptions on the initial data} 
Suppose that the initial data satisfy the following conditions.
\ \\
\begin{itemize}
    \item[(H4)] $M_{\varepsilon,1}(0)\in(-L,L)$, and $\varepsilon^{-2k_0}M^{c}_{\varepsilon,2k_0}(0)\leq{}C_1$, with $k_0$ large enough that \[r_L\left(1-\frac{1}{4^{k_0}}\right)-\max_{(x,t)\in[-L,L]\times[0,\infty)}m_\varepsilon(x,t):=\eta>0.\]
Moreover, $q_{\varepsilon}(x,0)\leq{}C_2\exp(-\beta|x|)$ for a positive constant $\beta$.

    \item[(H5)] $M_{\varepsilon,1}(0)=x_0+O(\varepsilon)>0$, for $x_0\in(-L,L)$.
    \item[(H6)] $\rho_m\leq{}\rho_\varepsilon(0)\leq{}\rho_M$ for positive constants $\rho_m$ and $\rho_M$. 
    \end{itemize}
\ \\    
    \begin{remark}
   The assumption (H4) ensures firstly that the trait-distribution is initially concentrated, and moreoever that a dissipative term appears in the evolution equation of $M_{\varepsilon,2k_0}$. Assumption (H5) requires that the mean trait is initially not far from a local minimum of the mortality rate. The last assumption (H6) specifies that the $\varepsilon$-independent bounds to the initial population size.
\end{remark}
Inspired from \cite{GHM}, we aim to show that the trait distribution $q_\varepsilon$ may be approximated, in the $L_1$ Wasserstein distance (see below for the definition), by a Gaussian distribution $g_\varepsilon$ defined as
\begin{equation}\label{eqn:approximate_distribution}
    g_\varepsilon(x,t)=\frac{1}{\sqrt{2\pi\varepsilon}}\exp\left(-\frac{(x-\bar{Z}_\varepsilon(t))^2}{2\varepsilon^2}\right),
\end{equation}
where  the approximate mean trait $\bar{Z}_\varepsilon$ solves the following equation
\begin{equation}
\label{eqn:approx_mean_trait}
\left\{
\begin{aligned}
\dot{\bar{Z}}_\varepsilon(t) &= -  \partial_{x}m_\varepsilon(\bar{Z}_\varepsilon(t)-X_\varepsilon(t),t), \\
\bar{Z}_\varepsilon(0) & = M_{\varepsilon,1}(0).  
\end{aligned}
\right.
\end{equation}
This may be recognised as the canonical equation in adaptive dynamics \cite{champagnat2002canonical}. The assumptions (H1)-(H6) are directly adapted from \cite{GHM}. The next assumptions are needed to deal with the time-dependence in the mortality function, which is new.
\ \\
\begin{itemize}
    \item[(H7)] $X_{\varepsilon}(0)=0$, $X_\varepsilon(t)\in{}C^{1}(\mathbb{R})$, $\int_{0}^{t}|\dot{X}_\varepsilon(s)|e^{-\frac{\eta}{\varepsilon^2}(t-s)}ds<\varepsilon{}L_X$, where  $L_X>0$ is a constant.
    \item[(H8)] There exists an $X_0(t)\in{}C^{1}(\mathbb{R})$, $m_0(x,t)\in{}C^{2,1}(\mathbb{R}^2)$, and $r_0(t)$ such that $\partial_{x}m_\varepsilon(x,t)\xrightarrow[\varepsilon\rightarrow{0}]{}\partial_{x}m_0(x,t)$ and $X_\varepsilon(t)\xrightarrow[\varepsilon\rightarrow{0}]{}X_0(t)$ uniformly in $t$, and in particular $|r_\varepsilon(t)-r_0(t)|\leq{}K_r\varepsilon.$
\end{itemize}
\ \\
Under the above assumption, the mean trait $\bar{Z}_\varepsilon(t)$ converges to $\bar{Z}_0(t)$, which solves
\begin{equation}
\label{eqn:limiting_problem}
\left\{
\begin{aligned}
\dot{\bar{Z}}_0(t) &= -  \partial_{x}m_0(\bar{Z}_0(t)-X_0(t),t), \\
\bar{Z}_0(0) & = x_0,  
\end{aligned}
\right.
\end{equation}

\begin{remark}
The control on the derivative in assumption (H7) of the moving optimal trait $X_\varepsilon(t)$ is general enough to handle, for instance $X_\varepsilon(t)=\cos\left(\frac{t}{\varepsilon}\right)$ which oscillates rapidly. However, it should be noted that this choice does not uniformly converge to any $X_0\in{}C^{1}(\mathbb{R})$, so (H8) does not apply for such a function. We could use an averaging procedure to capture the limiting behaviour, but for simplicity we restrict ourselves to those $X_\varepsilon(t)$ which do converge uniformly, as in (H8).
\end{remark}
We lastly introduce an assumption  to control the lag ${Y}_\varepsilon(t):=\bar{Z}_\varepsilon(t)-X_\varepsilon(t)$  which could cause the mean trait to stray from the local minimum $X_\varepsilon(t)$. 
%The limiting lag is given by $Y_0:=\bar{Z}_0(t)-X_0(t)$.
 \ \\
 \begin{itemize}
     \item[(H9)] There exists a constant $\gamma>0$ such that for any $\varepsilon>0$ we have $\sup_{t\geq{0}}|Y_\varepsilon(t)|<L-\gamma$.
 \end{itemize}
 
\subsection*{Wasserstein distances}
Let $\mu$ and $\nu$ be two probability measures on $\mathbb{R}^d$ with finite $p$-th moments for $p\in [1,\infty)$. The $L^p$-Wasserstein distance $W_p(\mu,\nu)$ between $\mu$ and $\nu$ is given by
\begin{equation}
 W_p(\mu,\nu)=\Big(\inf_{\pi\in \Pi(\mu,\nu)}\int_{\mathbb{R}^d\times \mathbb{R}^d}|x-y|^p d\pi(x,y)\Big)^{1/p},   
\end{equation}
where $\Pi(\mu,\nu)$ is the set of all couplings between $\mu$ and $\nu$, that is the set of all probability measures on $\mathbb{R}^d\times \mathbb{R}^d$ that have $\mu$ and $\nu$ as the first and second marginal respectively. The Wasserstein distances provide useful and flexible tools for quantifying the difference between two probability measures and have been used extensively in many branches of mathematics such as optimal transport, probability theory and partial differential equations. We refer the reader to Villani's monograph \cite{villani2021topics} for a detailed account of this topics.

The following theorem for the single-species model is the first main result of this paper. It shows that $q_\varepsilon(x,t)$ is quantitatively approximated by the Gaussian distribution $g_\varepsilon(x,t)$ in the Wasserstein distance $W_1$, and characterises the limiting behaviour of the total population $\rho_\varepsilon(t)$ in terms of the limiting mean trait $\bar{Z}_0$. This theorem extends \cite[Theorem 2]{GHM} to the case of time-dependent mortality and birth rates.

\subsection*{Main result on the single-species model} 
\begin{theorem}
\label{thm:main}
Assume (H1)--(H5). \\
(i) Let $\delta\in (0,1)$.  There exists a constant $K$ depending on the initial condition,   {$r_0(x,t) $ and $m_0(x,t)$} such that,  for any $\varepsilon$ sufficiently small, there holds 
\begin{equation}
\label{est-main-Was}
\sup_{t\in[0,+\infty)} W_1(q_{\varepsilon }(\cdot,t),g_{\varepsilon }(\cdot,t)) \leq  K \left( W_1(q_{\varepsilon,0},g_{\varepsilon,0}) + \varepsilon^{1-\delta} \right).
\end{equation}
with $q_{\varepsilon,0}$ and $g_{\varepsilon,0}$ being the respective values at time $t=0$ of $q_\varepsilon$ and $g_{\varepsilon}.$\\
(ii) Assume additionally (H6)--(H8) and let $\delta\in (0,1),$ and $\beta\in (1,2)$. Then, there exists a constant $K_\rho$ such that, for any $\varepsilon$ sufficiently small, we have
\begin{equation}
\label{as-rho}
| \rho_\varepsilon(t)-\rho(t)|\leq K_\rho\varepsilon^{1-\delta},\qquad \rho(t)=\frac{r_0(t)-m_0(\bar{Z}_0(t),t)}{\kappa},\quad \text{for all $t\in [\varepsilon^\beta,+\infty)$}
\end{equation}
  where $\bar{Z}_0$ solves \eqref{eqn:limiting_problem}. Moreover, as $\varepsilon\to 0$, $n_\varepsilon$ converges in $C\left((0,\infty) ; \mathcal{M}^1(\mathbb{R}) \right)$ to a measure $n$, which is given by 
 $$
 n(x,t)=\rho(t)\delta_{x-\bar{Z}_0(t)}.
 $$
\end{theorem}
In the above theorem, $\delta_{x-y}$ denotes the Dirac mass centred at $y$.
\subsection{Predator-Prey Model}
We now focus on the predator-prey model \eqref{eqn:evolving_prey_simplified}. In the following we introduce the variables $f^{*}=\sqrt{\frac{\gamma}{\kappa_2}}f$ and $h^{*}=\sqrt{\frac{\kappa_2}{\gamma}}h$ so that (once we omit the asterisk for notational convenience) the model reads
\begin{equation} \label{eqn:IP_standard_form}
\begin{cases}
\varepsilon^2 \partial_t   n_{1,\varepsilon}= r_1T_{{1,\varepsilon}}  [ n_{1,\varepsilon}]  
-   \left( \tilde{m}_{\varepsilon}(x,t)  + \kappa_1 \rho_{1,\varepsilon} \right) n_{1,\varepsilon} ,\\
\rho_{1,\varepsilon}(t)= \int_{\mathbb R}   n_{1,\varepsilon}(y,t)dy,\\
n_{1,\varepsilon}(x,0)=n_{{1,\varepsilon},0}(x), 
\end{cases}
\end{equation}
where $\tilde{m}_\varepsilon(x,t)=(F(x,\rho_{1,\varepsilon}(t),\bar{f}[q_{1,\varepsilon}](t))-\delta(x))\rho_{1,\varepsilon}(t)$ and we define $F:(x,I,C)\in\mathbb{R}^3\mapsto{}\mathbb{R}$ by $F(x,I,C)=\frac{f(x)C}{(1+hCI)^2}$.

Our aim is also to quantitatively derive the asymptotic limits for the mean trait and the population size in the regime of small variance $\varepsilon\ll 1$. Due to the additional nonlinearity and non-locality arising from the trait-dependent mortality rate, the analysis of this model is much more intricate than the single-species model. The key to overcome this challenge is to find suitable assumptions on the contact rate $f(x)$ and competition-modification $\delta(x)$ such that the hypotheses (H1)-(H9) are satisfied. To do this, we first need to identify $m_\varepsilon(x,t)$ and $X_\varepsilon(t)$ by considering the form of the limiting equations. It follows from \eqref{eqn:evolving_prey_simplified} that the population size and approximate mean trait respectively satisfy the following equations:
\begin{equation}\label{eqn:IP_population_size}
\varepsilon^{2}\frac{d\rho_{1,\varepsilon}}{dt}=\left(r_1-\int_{\mathbb{R}}\tilde{m}_\varepsilon(y,t)q_{1,\varepsilon}(y,t)dy-\kappa_1\rho_{1,\varepsilon}\right)\rho_{1,\varepsilon},
\end{equation}
and 
\begin{align*}
    \dot{\bar{Z}}_\varepsilon=&-\partial_{x}\tilde{m}_\varepsilon(\bar{Z}_\varepsilon,t)\\
    =&-(\partial_{x}F(x,\rho_{1,\varepsilon}(t),\bar{f}[q_{1,\varepsilon}](t))-\partial_{x}\delta(\bar{Z}_\varepsilon))\rho_{1,\varepsilon}(t).
\end{align*}
Then, formally, the approximations $M_{\varepsilon,1}(t)\approx{}\bar{Z}_{\varepsilon}(t)$ and $q_\varepsilon(x,t)\approx{}g_\varepsilon(x,t)$ give $\rho_{1,\varepsilon}\approx{}I(\bar{Z}_\varepsilon)$ where $I(x)$ is a positive solution to 
\[r_1-(F(x,I,f(x))^2+\kappa_1-\delta(x))I=0.\]
We will introduce assumptions to ensure $I(x)$ is locally unique. Presuming uniqueness for now, substituting $\rho_{1,\varepsilon}(t)=I(\bar{Z}_\varepsilon)$ back into $\tilde{m}_\varepsilon(x,t)$, and estimating $\bar{f}[q_{1,\varepsilon}](t)\approx{}f(\bar{Z}_\varepsilon)$, we find the effective equation for $\bar{Z}_\varepsilon$ is 
\begin{equation}\label{eq:trait_dynamics}
    \dot{\bar{Z}}_\varepsilon=-(\partial_{x}F(\bar{Z}_\varepsilon,I(\bar{Z}_\varepsilon),f(\bar{Z}_\varepsilon))-\partial_{x}\delta(\bar{Z}_\varepsilon))I(\bar{Z}_\varepsilon).
\end{equation}
These calculations inform the following assumptions.
\subsection*{Assumptions on the contact rate and competition strength}

The following are natural assumptions on the trait-dependent contact rate and competition strength.
\ \\
\begin{itemize}
    \item[(A0)] The functions $f(x)$ and $\delta(x)$ are $C^3(\mathbb{R})$. Also, $\kappa_1-\delta(x)>\kappa^{*}$ for some constant $\kappa^{*}>0$.
    \end{itemize}
\ \\
We here assume the trait-dependent contact rate and competition strength are sufficiently smooth for subsequent calculations. The last part of (A0) corresponds to the biological assumption that there is a positive competition-strength regardless of resource gain. 

Our next assumption ensures there is an equilibrium point of the trait dynamics (\ref{eq:trait_dynamics}), and that there is a unique stable manifold to which the population dynamics (\ref{eqn:IP_population_size}) are attracted. We first define the  following function which captures the approximate per captia growth rate of the prey: $G: (x,I)\in\mathbb{R} ^2\mapsto\mathbb{R}$ such that $G(x,I)=r_1-(F(x,I,f(x))^2+\kappa_1-\delta(x))I$. We assume
\ \\
\begin{itemize}
    \item[(A1)] There exists $x^{*}\in\mathbb{R}$ and constants $L_1,G^{*}>0$ such that for any $x$ satisfying $|x-x^{*}|<L_1$ we have $\partial_{I}G(x,I)<-G^{*}$, and the equation
$G(x,I)=0$ has a positive solution $I(x)$ which is continuous as a function of $x$. Moreover, $\partial_{x}F(x^{*},I(x^{*}),f(x^{*}))-\partial_{x}\delta(x^{*})=0$.
    \end{itemize}
\ \\
The assumptions on the $I$-derivative of $G$ ensures that $I(x)$ is associated to a stable equilibrium of the population dynamics. We specify that $I(x)$ changes continuously i.e. the root is non-degenerate as $F(x,I,f(x))$ and $\delta(x)$ vary in $x$ locally to the fixed point $x^{*}$. As a consequence, if $|x-x^{*}|<L_1$, there are constants $I_L,I_U>0$, depending on the $f$, $\delta$ and parameters other than $\varepsilon$, such that there is a unique solution $I(x)$ of $G(x,I)=0$ that satisfies $I_L<I(x)<I_U$.

Our next assumption ensures the convexity required in (H1) (and therefore the stability of the fixed point $x^{*}$ for the evolution equation of the mean trait), as well as the local uniqueness of the equilibrium point. 
%To this end, we introduce the parameterised function $h(x,I,C)=(F(x,I)C-\delta(x))I$, so that the mortality function may be written $\widetilde{M}(x,t)=h(x,\rho_{1,\varepsilon},\bar{F}[q_1,\varepsilon](t))$. 
We assume
 \ \\   
    \begin{itemize}
    \item[(A2)] 
    %For $x$ satisfying $|x-x^{*}|<L_1$, $|I-I(x^{*})|<L_2$, and $|C-F(x^{*},I(x^{*})|<L_3$ (where $L_2$ and $L_3$ are positive constants) we have $h''(x,I;C)>A_0$ for a constant $A_0>0$.
    For $x$ satisfying $|x-x^{*}|<L_1$ we have $\partial_{xx}F(x,I(x),f(x))-\partial_{xx}\delta(x)>A_0$ for a constant $A_0>0$.
    
\end{itemize} 
\ \\
Based on (A2), it is clear that, by continuity, there exist positive constants $L_2$ and $L_3$ such that for $|x-x^{*}|<L_1$, $|I-I(x)|<L_2$, and $|C-f(x)|<L_3$ we have that $\partial_{xx}F(x,I,C)-\partial_{xx}\delta(x)>A_0$ (potentially taking a smaller $A_0$ and $L_1$). This implies, there is a unique minimum in $x$ of $F(x,I,C)-\delta(x)$ provided that  $x$, $I$, and $C$ respect those bounds. We denote this as $x_m(I,C)$. The next assumption on the initial data arises since we require that $X_\varepsilon(t)=x_m(\rho_{1,\varepsilon}(t),\bar{f}[q_{1,\varepsilon}](t))$ remains unique for $t>0$. Moreover, we note that $X_\varepsilon(t)$ depends on $\rho_{1,\varepsilon}(t)$, and so the derivative depends on $\frac{d\rho_{1,\varepsilon}(t)}{dt}$. Therefore, in order to control the magnitude of this derivative, we need to assume $\rho_\varepsilon(t)$ is initially close to $I(x_0)$. 
\ \\
\begin{itemize}
    \item[(A3)] We assume that
 $M_{\varepsilon,1}=x_0+O(\varepsilon)$ where $|x_0-x^{*}|<L_1$ and {$|\rho_{1,\varepsilon}(0)-I(x_0)|<K_3\varepsilon$} for a constant $K_3>0$.
\end{itemize}
\ \\
Since we will also need to control $\rho_{1,\varepsilon}(t)$, the next assumption is introduced to control the growth rate of $m_\varepsilon(x,t)$ for large $|x|$ in (H3).
\ \\
\begin{itemize}
    \item[(A4)] For all $x\in\mathbb{R}$ we have $|f''(x)|+|\partial_{xx}\delta(x)|<A_m(1+|x|^{p})$ for a constant $A_m>0$ and exponent $p\in\mathbb{N}$, $p>1$.
\end{itemize}
\ \\
With these assumptions in mind, we shift the $x$-coordinate via a change of variables $\tilde{x}=x+x_{m}(\rho_{1,\varepsilon}(0),\bar{f}[q_{1,\varepsilon}](0))$ to ensure that the minimum begins at $0$. For simplicity of notations we will now omit the tilde, write $x_{m,0,\varepsilon}=x_m(\rho_{1,\varepsilon}(0),\bar{f}[q_{1,\varepsilon}](0))$, and will at times suppress the arguments of $x_m(\rho_{1,\varepsilon}(t),\bar{f}[q_{1,\varepsilon}](t))$.
We may now write \[X_\varepsilon(t)=x_m-x_{m,0,\varepsilon},\] 
\[\tilde{m}_\varepsilon(x,t)=\left(F(x+x_{m,0,\varepsilon},\rho_{1,\varepsilon},\bar{f}[q_{1,\varepsilon}](t))-\delta(x+x_{m,0,\varepsilon})\right)\rho_{1,\varepsilon}(t),\]
 and 
\begin{equation}\label{eqn:IP_mortality}
{m}_\varepsilon(x,t)=\left(F(x+x_{m},\rho_{1,\varepsilon},\bar{f}[q_{1,\varepsilon}](t))-\delta(x+x_{m})\right)\rho_{1,\varepsilon}(t).
    \end{equation}
We further propose that limiting mortality function and optimal trait are given by
\begin{align*}
X_0(t)&=x_m(\rho_{1}(t),f(\bar{Z}_0))-x_m(\rho_{1}(0),f(x_0)),\\
m_0(x,t)&=\left(F(x+x_m(\rho_{1}(t),f(\bar{Z}_0)),\rho_1,f(\bar{Z}_0))-\delta(x+x_m(\rho_{1,\varepsilon}(t),\bar{F}[q_{1,\varepsilon}](t)))\right)\rho_{1}(t),
\end{align*}
where the limiting mean trait $\bar{Z}_0(t)$ and total population size $\rho_1(t)$ satisfy
\begin{align*}
   \dot{\bar{Z}}_0&=-\partial_{x}m_0(\bar{Z}_0-X_0(t),t),\\
   \rho_1(t)&=I(\bar{Z}_0).
\end{align*}
Here $I(x)$ is the unique positive root of $G(x,I)$ that satisfies $I_L<I(x)<I_U$. We have will show that $|\bar{Z}_0-x^{*}|<L_1$ which ensures the $I(\bar{Z}_0)$ is well defined for all $t\geq{0}$. Note that substituting $\rho_1(t)$ gives an autonomous differential equation for $\bar{Z}_0$.

Assumption (H4) can be relabelled as (A5) for this section and requires no modification. For the sake of compactness, we include the exponential bound on the initial distribution as well.
\ \\
\begin{itemize}
\item[(A5)] $M_{\varepsilon,1}(0)\in(-L,L)$, and $\varepsilon^{-2k_0}M^{c}_{\varepsilon,2k_0}(0)\leq{}C_1$, with $k_0$ large enough that \[r_1\left(1-\frac{1}{4^{k_0}}\right)-\max_{(x,t)\in[-L,L]\times[0,\infty)}m_0(x,t):=\eta>0.\]
Moreover, $q_{1,\varepsilon}(x,0)\leq{}C_2\exp(-\beta|x|)$ for a constant $\beta>0$.
\end{itemize}
The purpose of (H2) in the single-species model is to ensure the population does not go extinct. This is no longer necessary for this specific model due to the form of the mortality function (\ref{eqn:IP_mortality}), since per capita mortality vanishes for a decreasing population. The irrelevance of (H2) is made explicit in Section \ref{sec: proof of 2nd thm}, during the proof of the main result for the predator-prey model.

\subsection*{Main result on the prey-predator model}
We are now ready to state the second main result of this paper. The following theorem quantitatively establishes the asymptotic dynamics for the mean trait and population size in the regime of small variance.
%Can include distribution convergence too
%\begin{theorem}\label{thm:IP_main}
% Assume (A0)-(A5). Let $\delta\in(0,1)$, and $\beta\in(1,2)$. Then, there exists a constanst $K,K_\rho$ such that for any sufficiently small $\varepsilon$, we have
% \begin{equation}
% \label{est-main-Was-pred-prey}
% \sup_{t\in[0,+\infty)} W_1(q_{1,\varepsilon }(\cdot,t),g_{1,\varepsilon }(\cdot,t)) \leq  K \left( W_1(q_{1,\varepsilon,0},g_{1,\varepsilon,0}) + \varepsilon^{1-\delta} \right),
% \end{equation}
% and
% \begin{equation}
%     |\rho_{1,\varepsilon}(t)-\rho_1(t)|\leq{}K_\rho\varepsilon^{1-\delta}, \text{ for all } t\in[\varepsilon^\beta,\infty),
% \end{equation}
% As $\varepsilon\rightarrow{0}$, $n_{1,\varepsilon}$ converges in $C((0,\infty);\mathcal{M}^{1}(\mathbb{R}))$ to a measure $n_1$ given by
% \[n_1(x,t)=\rho_1(t)\delta_{x-\bar{Z}_{0}(t)}.\]
% Here the limiting mean trait $\bar{Z}_0(t)$ and total population size $\rho_1(t)$ satisfy
% \begin{align*}
%    \begin{cases}
%      \dot{\bar{Z}}_0=-\partial_{x}m_0(\bar{Z}_0-X_0(t),t),\\
%      \bar{Z}_0(t)=x_0,
%    \end{cases} ~\text{and}~~ \rho_1(t)=I(\bar{Z}_0). 
% \end{align*}
% Furthermore, $|\bar{Z}_0-x^{*}|<L_1$, where $L_1>0$ is the same as in assumptions (A1) and (A2), and $I(x)$ is the unique positive solution of $r_1-(F(x,I,f(x))^2+\kappa+1-\delta(x))I=0$ that satisfies $I_L<I(x)<I_U$ where $I_L$ and $I_U$ are constants determined by assumption (A1).

% \end{theorem}
\begin{theorem}\label{thm:IP_main}
Assume (A0)-(A5). Let $\delta\in(0,1)$, and $\beta\in(1,2)$. Then, there exists a constant $K_\rho$ such that for any sufficiently small $\varepsilon$, we have
\begin{equation}
    |\rho_{1,\varepsilon}(t)-\rho_1(t)|\leq{}K_\rho\varepsilon^{1-\delta}, \text{ for all } t\in[\varepsilon^\beta,\infty),
\end{equation}
and as $\varepsilon\rightarrow{0}$, $n_{1,\varepsilon}$ converges in $C((0,\infty);\mathcal{M}^{1}(\mathbb{R}))$ to a measure $n_1$ given by
\[n_1(x,t)=\rho_1(t)\delta_{x-\bar{Z}_{0}(t)}.\]
Here the limiting mean trait $\bar{Z}_0(t)$ and total population size $\rho_1(t)$ satisfy
\begin{align*}
   \begin{cases}
     \dot{\bar{Z}}_0=-\partial_{x}m_0(\bar{Z}_0-X_0(t),t),\\
     \bar{Z}_0(t)=x_0,
   \end{cases} ~\text{and}~~ \rho_1(t)=I(\bar{Z}_0). 
\end{align*}
Furthermore, $|\bar{Z}_0-x^{*}|<L_1$, where $L_1>0$ is the same as in assumptions (A1) and (A2), and $I(x)$ is the unique positive solution of $r_1-(F(x,I,f(x))^2+\kappa+1-\delta(x))I=0$ that satisfies $I_L<I(x)<I_U$ where $I_L$ and $I_U$ are constants determined by assumption (A1).

\end{theorem}
This theorem shows that the differential equation satisfied by the limiting mean trait, the canonical equation of adaptive dynamics, may be expressed as an autonomous equation in $\bar{Z}_0(t)$. The shape of the mortality rate now takes into account the eco-evolutionary dynamics of the system and provides an effective "fitness landscape" along which the mean trait $\bar{Z}_0(t)$ evolves.

In Section \ref{sec: numerics} we provide concrete examples and perform numerical simulations to demonstrate the main analytical results.

\subsection*{The main idea of the proofs: a moment-based method} The proofs of Theorems \ref{thm:main} and \ref{thm:IP_main} are presented in Sections \ref{sec: proof of 1st theorem} and \ref{sec: proof of 2nd thm} respectively. We adopt and generalize the moment-based method developed in \cite{GHM}. The moment-based method transforms an integro-differential equation (or a partial differential equation) into a set of ordinary differential equations for the moments. The key challenge is that these equations are not closed, that is, the equation for a moment often depends on higher moments (in other words, one obtains an infinite set of equations). In most situations where there is no analytical relations between the moments, one simply truncates the systems by setting all sufficiently high moments to zero to obtain an approximate finite system.  It turns out that the moment-based method is suitable for analyzing the asymptotic limits of the infinitesimal equations \eqref{eqn:main} and \eqref{eqn:evolving_prey_simplified} since one can explicitly quantify their higher moments. This is the key technical step in the proofs of the theorems, which will be carried out in the next section. We also refer the reader to \cite{hadeler2004langevin,villa2024reducing} for some interesting applications of the moment-based method for similar models in mathematical biology.

\section{Moment estimates and proof of \cref{thm:main}}
\label{sec: proof of 1st theorem}
In this section, we will prove Theorem \ref{thm:main} which quantitatively establishes the asymptotic analysis for the trait distribution and the total population size of the single-species model \eqref{eqn:main}. As mentioned earlier, we will adopt and generalize the approach in \cite{GHM} to the setting of time-dependent mortality and birth rates of this paper. 

The following estimates of the central moments defined in \eqref{eq: central moments} will play a crucial role in the subsequent analysis.
\begin{proposition}
\label{thm:MomentEstimates}
Assume (H1)--(H9) and $\delta \in (0,1).$   For any $\varepsilon$ sufficiently small,  there exist large enough positive constants $K_0,K_1$ and $K_2$,  depending on the parameters $C_1,r,L$, such that:
\begin{align}
 \label{eqn:M2kclose} M^{c}_{\varepsilon,2k_0}(t) &  \leq  K_{2} \varepsilon^{2k_0}H_{2,\varepsilon}(t), \\
\label{eqn:M2close} 
 \left | M_{\varepsilon,2}^{c}(t) -  \varepsilon^2  \right| & \leq  K_1\varepsilon^2 H_{1,\varepsilon}(t), \\
\label{eqn:M1close} \left| M_{\varepsilon,1}(t) - \bar{Z}_{\varepsilon}(t) \right|  & \leq  K_0 \varepsilon^{1-{\delta}}. 
\end{align}
for any $t \in \mathbb R^+$,  where $H_{2,\varepsilon}(t)\leq{}2$, and eventually $H_{1,\varepsilon}(t)\leq{}2\varepsilon^{1-\delta}$ if $t$ is large enough.
\end{proposition} 
The proof of this proposition requires some preliminary calculations. We will need to work with a shifted central moment and distribution. To this end, we define  
\[
Q_\varepsilon(x,t)=q_\varepsilon(x+X_\varepsilon(t),t)~\text{and}~\widetilde{M}_{\varepsilon,1}=M_{\varepsilon,1}-X(t).
\]
We obtain from \eqref{eqn:distribution} the following equation for $Q_\varepsilon$:
\begin{equation} \label{eqn:distribution_shift}
\varepsilon^2 \partial_t Q_\varepsilon-\varepsilon^2\partial_{x}Q_\varepsilon{}\dot{X}_\varepsilon(t) = r_\varepsilon(t)\left(\widetilde T_{\varepsilon}[Q_\varepsilon] - Q_\varepsilon \right) 
-  \left( m_\varepsilon(x,t)  - \int_{\mathbb R} m_\varepsilon(x,t)Q_\varepsilon dx \right) Q_\varepsilon.
\end{equation}
 Note that $\widetilde{M}_{\varepsilon,1}=\int_{}xQ_\varepsilon(x,t)dx$ and that \begin{align*}
     \int_{\mathbb{R}}(x-\widetilde{M}_{\varepsilon,1})^kQ_\varepsilon{}dx&=\int_{\mathbb{R}}(x-{M}_{\varepsilon,1})^kq_\varepsilon{}dx= M_{\varepsilon,k}^c.
 \end{align*}
Next, we recall the notation for the remainder from \cite{GHM}; for a $C^2(\mathbb{R})$ function $f$ the remainder in the first order Taylor-Lagrange is given by
    \[r^{f}[X](x)=\int_{0}^{1}(1-\sigma)f''(X+\sigma(x-X)),\]
    and we may write
    \[f(x)=f(X)+(x-X)f'(X)+(x-X)^2r^f[X](x).\]
Since we will compute expansions in the $x$ component of functions which depend on $x$ and $t$, given any function $f:\mathbb{R}^2\to{}\mathbb{R}$ we define $f_t(x)=f(x,t)$; i.e. $f_t$ are a family of functions each of which depends only on $x$.  We note here that we may write 
$m_t(x)=m_\varepsilon(x,t)$ as
\[m_t(x)=m_t(\widetilde{M}_{\varepsilon,1})+(x-\widetilde{M}_{\varepsilon,1})m_t'(\widetilde{M}_{\varepsilon,1})+(x-\widetilde{M}_{\varepsilon,1})^2r^{m_t}[\widetilde{M}_{\varepsilon,1}](x),\]
and it follows from (H3) and the fact that $p>1$ that $|r^{m_t}[\widetilde{M}_{\varepsilon,1}](x)|\leq{}A_m(1+|\widetilde{M}_{\varepsilon,1}|^p+|\widetilde{M}_{\varepsilon,1}-x|^p)$.
\subsection{First Shifted Moment}  
We begin by multiplying \eqref{eqn:distribution_shift} by $x$ and integrating. We arrive at
\[
\varepsilon^2 (\dot{\widetilde{M}}_{\varepsilon,1}+\dot{X}_\varepsilon(t))= -   \left(  \int_{\mathbb R} m_\varepsilon(x,t)xQ_\varepsilon(x,t){\rm d}x  - I_{\varepsilon,m} \widetilde{M}_{\varepsilon,1} \right),
\]
    where $I_{\varepsilon,m}(t)=\int_{}m_\varepsilon(x,t)q_\varepsilon(x,t)dx$. This can be expanded as 
    \begin{equation}\label{eqn:Iem}
        I_{\varepsilon,m}(t)=m_t(\widetilde{M}_{\varepsilon,1})+\int_{0}^{1}\int_\mathbb{R}(1-\sigma)m_t''(\widetilde{M}_{\varepsilon,1}+(x-\widetilde{M}_{\varepsilon,1}\sigma))d\sigma{dx}.
    \end{equation}
    A straightforward calculation, using (H3), then shows
    \[|I_{\varepsilon,m}-m_{t}(\widetilde{M}_{\varepsilon,1})|<CA_m\left((1+|\widetilde{M}_{\varepsilon,1}|^p)M_{\varepsilon,2}^c+\widetilde{M}_{\varepsilon,2+p}^{|c|}\right)\]
We now let $\psi(x,t)=xm_\varepsilon(x,t)$ and, for the purposes of the Taylor expansions, we let $\psi_{t}(x)=\psi(x,t$) and $m_t(x)=m_\varepsilon(x,t)$. From this we can write the first term as 
\begin{align*}
\int_{\mathbb{R}}m_t(x)xQ_\varepsilon(x,t)dx=m_t(\widetilde{M}_{\varepsilon,1})\widetilde{M}_{\varepsilon,1}+\int_\mathbb{R}(x-M_{\varepsilon,1})^2r^\psi[\widetilde{M}_{\varepsilon,1}](x)q(x,t)dx.
\end{align*}
Using the fact that $\psi_t''(x)=xm_t''(x)+2m_t'(x)$, and the integral form of the remainder, we arrive at
    \[\varepsilon^2(\widetilde{M}_{\varepsilon,1}+\dot{X}_\varepsilon(t))+m_{t}'(\widetilde{M}_{\varepsilon,1})M_{\varepsilon,2}^c=-F_1.\]
 
    \begin{equation}\label{eqn:F1bound}
        |F_1|\leq{}CA_m\left((1+|\widetilde{M}_{\varepsilon,1}|^p){M}_{\varepsilon,3}^{|c|}+{M}_{\varepsilon,3+p}^{|c|}\right).
    \end{equation}

      \subsection{Second Moment}  
    Computations for the second moment is identical to that in \cite{GHM} apart from an extra term but for completeness we provide the full calculation here. We firstly compute
    \begin{align*}
        \varepsilon^2\dot{M}_{\varepsilon,2}^c=&-2\varepsilon^2\dot{\widetilde{M}}_{\varepsilon,2}\int_{\mathbb{R}}(x-\widetilde{M}_{\varepsilon,1})Q_\varepsilon(x,t)dx+\varepsilon^2\dot{X}_\varepsilon\int_{\mathbb{R}}(x-\widetilde{M}_{\varepsilon,1})^2\partial_{x}Q_\varepsilon(x,t)\\&+\varepsilon^2\int_{\mathbb{R}}(x-\widetilde{M}_{\varepsilon,1})^2\partial_{t}q_\varepsilon(x+X_\varepsilon,t).
    \end{align*}
    The first term vanishes, by the definition of $\widetilde{M}_{\varepsilon,1}$ as does the the additional term $-\varepsilon^2\dot{X}_\varepsilon(t)\int_{\mathbb{R}}(x-\widetilde{M}_{\varepsilon,1})^2\partial_{x}Q_\varepsilon{}dx$ which results from the time-dependent optimum. To thee this, observe that, by an application of the comparison principle to \eqref{eqn:distribution_shift}, making use of the initial bound on $q_\varepsilon$ from (H4), the function $\partial_{x}Q_\varepsilon(x,t)$ decays exponentially in $|x|$ for all $t>0$. Then integration by parts and the definition of $\widetilde{M}_{\varepsilon,1}$ gives the conclusion.
    
    We focus on the last term now. Expanding it using \eqref{eqn:distribution} yields
    \begin{align*}
        \varepsilon^2\int_{\mathbb{R}}(x-\widetilde{M}_{\varepsilon,1})^2\partial_{t}q_\varepsilon(x+X_\varepsilon,t)=&r_\varepsilon(t)\int_\mathbb{R}(x-M_{\varepsilon,1})^2(\tilde{T}[q_\varepsilon]-q_\varepsilon)\\&
        -\left(\int_{\mathbb{R}}\tilde{m}_\varepsilon(x,t)(x-M_{\varepsilon,1})^2q_\varepsilon(x,t)dx-I_{\varepsilon,m}(t)\int_{\mathbb{R}}(x-M_{\varepsilon,1})^2q_\varepsilon(x,t)dx\right).
    \end{align*}

The term $\tilde{T}_\varepsilon[q_\varepsilon]$ satisfies the same algebra as in \cite{GHM}. Namely, we have
\begin{align*}
    \int_{\mathbb{R}}(x&-M_{\varepsilon,1})^2\tilde{T}[q_\varepsilon](x)dx\\
    &=\int_{\mathbb{R}}\int_{\mathbb{R}}\int_{\mathbb{R}}\left(\left(x-\frac{y-y'}{2}\right)+\left(\frac{y+y'}{2}-M_{\varepsilon,1}\right)\right)^2\Gamma_\varepsilon\left(x-\frac{y+y'}{2}\right)q_\varepsilon(y,t)q_\varepsilon(y',t)dydy'dx\\
    &=\frac{\varepsilon^2}{2}+\int_{\mathbb{R}}\int_{\mathbb{R}}\left(\frac{y-M_{\varepsilon,1}}{2}+\frac{y'-M_{\varepsilon,1}}{2}\right)^2q_\varepsilon(y,t)q_\varepsilon(y',t)dydy'\\
    &=\frac{\varepsilon^2}{2}+\frac{M_{\varepsilon,2}^c}{2}.
\end{align*}
    Using this identity, we obtain:
\begin{align*}
    \varepsilon^2\dot{M}_{\varepsilon,2}^c=r_\varepsilon(t)\left(\frac{\varepsilon^2}{2}-\frac{M_{\varepsilon,2}^c}{2}\right)-\left(\int_{\mathbb{R}}{m}_\varepsilon(x,t)(x-\widetilde{M}_{\varepsilon,1})^2Q_\varepsilon(x,t)dx-I_{\varepsilon,m}M_{\varepsilon,2}^c\right).
\end{align*}
We may expand the first term in the brackets as
\[\int_{\mathbb{R}}m_\varepsilon(x,t)(x-\widetilde{M}_{\varepsilon,1})^2Q_\varepsilon(x,t)dx=m_t(\widetilde{M}_{\varepsilon,1})M_{\varepsilon,2}^c+m_t'(\widetilde{M}_{\varepsilon,1})M_{\varepsilon,3}^c+\int_\mathbb{R}(x-\widetilde{M}_{\varepsilon,1})^2r^m[\widetilde{M}_{\varepsilon,1}](x)Q_\varepsilon(x,t)dx.\]
Combining the above expression with the expansion  for $I_{\varepsilon,m}$ \eqref{eqn:Iem} shows that
\begin{align*}
    \int_{\mathbb{R}}{m}_\varepsilon(x,t)&(x-\widetilde{M}_{\varepsilon,1})^2Q_\varepsilon(x,t)dx-I_{\varepsilon,m}M_{\varepsilon,2}^c=m'_t(\widetilde{M}_{\varepsilon,2})M_{\varepsilon,3}^c\\&+\int_0^1\int_\mathbb{R}(1-\sigma)m_t''(\widetilde{M}_{\varepsilon,1}+(x-\widetilde{M}_{\varepsilon,1})\sigma)((x-\widetilde{M}_{\varepsilon,1})^4-(x-\widetilde{M}_{\varepsilon,1})^2M_{\varepsilon,2^c})d\sigma{}Q_\varepsilon(x,t)dx.
\end{align*}

Denoting the RHS of the above as $F_2$  we obtain 
\begin{equation}\label{eqn:second_shift_moment}\varepsilon^2\dot{{M}}_{\varepsilon,2}^c+\frac{r_\varepsilon(t)}{2}{M}_{\varepsilon,2}^c=r_\varepsilon(t)\frac{\varepsilon^2}{2}+F_2,
\end{equation}
where we may use (H3) to bound $F_2$ as follows
\begin{equation*}
  F_2\leq{}CA_m(1+|\widetilde{M}_{\varepsilon,1}|+|\widetilde{M}_{\varepsilon,1}|^p+|\widetilde{M}_{\varepsilon,1}|^{p+1})(|M_{\varepsilon,2}|^2+M_{\varepsilon,3}^{|c|}+M_{\varepsilon,4}^{c}+M_{\varepsilon,4+p}^{c}+M_{\varepsilon,2}M_{\varepsilon,2+p}).  
\end{equation*}
Our expression of $F_2$ is slightly different to the expression in \cite{GHM}. We prefer to express it as a product of a sum of terms involving only the first shifting moment, and a sum of terms involving higher order moments. The evolution equation \eqref{eqn:second_shift_moment} is identical if we replace $r$ by $r_\varepsilon(t)$.
\subsection{Higher Order Moments}  
We compute 
\begin{align*}
\varepsilon^2{}\dot{M}_{\varepsilon,2k}=&\int_{\mathbb{R}}(x-{M}_{\varepsilon,1})^{2k}\varepsilon^2\partial_t{}q_{\varepsilon}(x,t)-2k\varepsilon^2\dot{\widetilde{M}}_{\varepsilon,1}\int_\mathbb{R}(x-\widetilde{M}_{\varepsilon,1})^{2k-1}Q_\varepsilon(x,t)dx\\&-\varepsilon^2\dot{X}_\varepsilon(t)\int_{\mathbb{R}}(x-M_{\varepsilon,1})^{2k}\partial_{x}Q_\varepsilon(x,t).
\end{align*}
Similarly to the second moment, compared to \cite{GHM} the evolution equation for the higher order even moment  $M_{\varepsilon,2k}^c$ has an additional term on the right hand side.  We can write this term as
\[\varepsilon^2{}\dot{X}_\varepsilon(t)\int_{\mathbb{R}}(x-\widetilde{M}_{\varepsilon,1})^{2k}\partial_{x}Q_\varepsilon{}dx=\varepsilon^2{2k}\dot{X}_\varepsilon(t)M_{\varepsilon,2k-1}.\] 
We now control the first term. We write
\begin{align*}
    \int_{\mathbb{R}}(x-{M}_{\varepsilon,1})^{2k}\varepsilon^2\partial_t{}q_{\varepsilon}(x,t)=&r_\varepsilon(t)\left(\int_{\mathbb{R}}(x-M_{\varepsilon,1})^{2k}\tilde{T}[q_\varepsilon]-M_{\varepsilon,2k}\right)\\
    &-\int_{\mathbb{R}}\tilde{m}_\varepsilon(x,t)(x-M_{\varepsilon,1})^{2k}q_\varepsilon(x,t)dx+I_{\varepsilon,m}M_{\varepsilon,2k}\\
    &+2k(F_1+\partial_xm_\varepsilon(\widetilde{M}_{\varepsilon,1},t)M_{\varepsilon,2}^c)M_{\varepsilon,2k-1}^{c}.
\end{align*}

We can again make use of the algebra satisfied by $\tilde{T}$. We write
\begin{align*}
    \int_{\mathbb{R}}(x-M_{\varepsilon,1})^{2k}\tilde{T}[q_\varepsilon]&=\int_\mathbb{R}\int_\mathbb{R}\int_\mathbb{R}\left(\left(x-\frac{y+y'}{2}\right)+\left(\frac{y+y'}{2}-M_{\varepsilon,1}\right)\right)^{2k}\Gamma\left(x-\frac{y+y'}{2}\right)q_\varepsilon(y,t)q_\varepsilon(y',t)dydy'dx.
\end{align*}
We then expand the polynomial $\left(\left(x-\frac{y+y'}{2}\right)+\left(\frac{y+y'}{2}-M_{\varepsilon,1}\right)\right)^{2k}$, and integrate out $\Gamma$. We also use the fact that the $l$th moment of $\Gamma_\varepsilon$ is of the form $\sigma_l\varepsilon^{2l}$ to eventually obtain
\begin{align*}
    \int_{\mathbb{R}}(x-M_{\varepsilon,1})^{2k}\tilde{T}[q_\varepsilon]=\frac{2}{4^k}M_{\varepsilon,2k}^c+\sum_{l=0}^{k-1}\sum_{j=0}^{2l}\sigma_{k-l}\varepsilon^{2(k-l)}\frac{1}{4^l}{2k\choose2l}{2l\choose{j}}M_{\varepsilon,2l-j}^cM_{\varepsilon,j}^c+\sum_{j=2}^{2k-2}\frac{1}{4^k}{2k\choose{j}}M_{\varepsilon,2k-j}M_{\varepsilon,j}^c.
\end{align*}
From the above, we get
\begin{equation}\label{eqn:HigherOrderMoments}
    \varepsilon^2{}\dot{M}_{\varepsilon,2k}^c+\left[r_\varepsilon(t)\left(1-\frac{1}{4^k}\right)-m_\varepsilon(\widetilde{M}_{\varepsilon,1},t)\right]M_{\varepsilon,2k}^c\leq{}F_{2k}+\varepsilon^2{2k}|\dot{X}_\varepsilon(t)M_{\varepsilon,2k-1}^{c}|.
\end{equation}
where
\begin{align*}
    F_{2k}\leq{}&2k(|F_1|+CA_m(|\widetilde{M}_{\varepsilon,1}|+|\widetilde{M}_{\varepsilon,1}|^{p+1})M_{\varepsilon,2}^c)|M_{\varepsilon,2k-1}^c+CA_mM_{\varepsilon,2k}^c\left((1+|\widetilde{M}_{\varepsilon,1}|^p)M_{\varepsilon,2}^c+M_{\varepsilon,2+p}^{|c|}\right)\\
    &+\sum_{l=0}^{k-1}\sum_{j=0}^{2l}\sigma_{k-l}\varepsilon^{2(k-l)}\frac{1}{4^l}{2k\choose2l}{2l\choose{j}}M_{\varepsilon,2l-j}^cM_{\varepsilon,j}^c+\sum_{j=2}^{2k-2}\frac{1}{4^k}{2k\choose{j}}M_{\varepsilon,2k-j}M_{\varepsilon,j}^c.
\end{align*}

\subsection{Proof of \cref{thm:MomentEstimates}}
Having completed preliminary calculations for all moment equations, we proceed similarly to \cite{GHM} and  define \[T_\varepsilon=\sup\{t\in[0,\infty):\eqref{eqn:M2kclose}-\eqref{eqn:M1close} \text{ hold for } t\in[0,t]\},\]
and argue that we can choose $K_0$,  $K_1$, $K_2$, $H_{1,\varepsilon}$ and $H_{2,\varepsilon}$ such that $T_\varepsilon=\infty$, similarly to \cite{GHM}. We show such a choice is possible for $K_0, K_1, K_2$ in the following lemma.
 \begin{lemma}
 \label{prop:InfiniteT_eps} 
 Assume (H1)-(H9). There are constants $K_0$,$K_1$, and $K_2$, and functions $H_{i,\varepsilon}(t)$ such that for a given $\delta\in(0,1)$ and $\varepsilon$ depending on  $\delta$ we have
 \begin{itemize}
     \item[(i)] $T_\varepsilon>0$
     \item[(ii)] if $T_\varepsilon\in(0,\infty)$ we have\begin{align}
 \label{eqn:M2kcloseT} M^{c}_{\varepsilon,2k_0}(t) &  <  K_{2} \varepsilon^{2k_0}H_{2,\varepsilon}(t), \\
\label{eqn:M2closeT} 
 \left | M_{\varepsilon,2}^{c}(t) -  \varepsilon^2  \right| & <  K_1\varepsilon^2 H_{1,\varepsilon}(t), \\
\label{eqn:M1closeT} \left| M_{\varepsilon,1}(t) - \bar{Z}_{\varepsilon}(t) \right|  & <  K_0 \varepsilon^{1-{\delta}},
\end{align}
 \end{itemize}
 where $H_{2,\varepsilon}(t)\leq{}3$ for $t\in[0,T_\varepsilon]$, and $H_{1,\varepsilon}(t)\leq{}3\varepsilon^{1-\delta}$ if $t$ is large enough.
     
 \end{lemma}

 \begin{proof}

 We first check that $\widetilde{M}_{\varepsilon,1}<L$ for $t\in[0,T_\varepsilon]$.   We have that
 $|\widetilde{M}_{\varepsilon,1}-Y_\varepsilon|\leq{}K_0\varepsilon^{1-\delta} \text{ on } t\in[0,T_\varepsilon],$
  which implies 
  \[
  |\widetilde{M}_{\varepsilon,1}|<|Y_\varepsilon|+K_0\varepsilon^{1-\delta}.
  \] 
From (H9), for sufficiently small $\varepsilon$, we have $|Y_\varepsilon|<L-\gamma$, and so $|\widetilde{M}_{\varepsilon,1}|<L$ for all $t\in[0,T_\varepsilon]$.

We may without loss of generality set $H_{1,\varepsilon}(0)=1$, and $H_{2,\varepsilon}=2$. For the first part, we choose $K_0$, $K_1$, and $K_2$ large enough with respect to the initial data. As in \cite{GHM}, this requires first choosing $K_1$ depending upon $C_1$ from the estimate on the initial data (H4), choosing $K_2$ depending on $K_1$ and $C_1$ and finally $K_0$ depending on $K_1$ and $L$.
 
 To show the second part, we will assume that $1+H_{1,\varepsilon}(t)\leq{}H_{2,\varepsilon}(t)$, which by the positivity of $H_{1,\varepsilon}(t)$ also requires that $H_{2,\varepsilon}(t)\geq{}1$. We  estimate the intermediate moments using the H\"{o}lder inequality,
\begin{align*}
M_{\varepsilon,l}^{|c|}&\leq{}M_{\varepsilon,2}^\frac{l-2}{2k_0-2}M_{\varepsilon,2k_0}^\frac{2k_0-l}{2k_0-2}\\
&\leq{}\varepsilon^{l}K_{2}^\frac{l-2}{2k_0-2}K_{1}^\frac{2k_0-l}{2k_0-2}H_{2,\varepsilon}(t)^\frac{l-2}{2k_0-2}(1+H_{1,\varepsilon}(t))^\frac{2k_0-l}{2k_0-2}\\
&\leq{}2\varepsilon^{l}K_{2}^\frac{l-2}{2k_0-2}K_{1}^\frac{2k_0-l}{2k_0-2}H_{2,\varepsilon}(t)^\frac{l-2}{2k_0-2}(1+H_{1,\varepsilon}(t))^\frac{2k_0-l}{2k_0-2}\\
&\leq{}\varepsilon^{l}K_{2}H_{2,\varepsilon}(t),
\end{align*}
where in the last line we assume $K_2$ is sufficiently larger than $K_1$.

 The computation of the second moment yields the equation 
\[\begin{cases}
    \dot{M}_{\varepsilon,2}+\frac{r_\varepsilon(t)}{2\varepsilon^2}M^{c}_{\varepsilon,2}=\frac{r_\varepsilon(t)}{2}+\frac{1}{\varepsilon^2}
F_2,\\
M_{\varepsilon,2}^{c}=\alpha_{\varepsilon,0}\varepsilon^2\end{cases},\]
whose solution may be written
\begin{align}\label{eqn:Meps2}
    M_{\varepsilon,2}(t)=&\varepsilon^2+\varepsilon^2(\alpha_{\varepsilon,0}-1)\exp\left(-\frac{1}{2\varepsilon^2}\int_{0}^{t}r_\varepsilon(s)ds\right)\nonumber\\
    &+\frac{1}{\varepsilon^2}\int_{0}^{t}F_2(s)\exp\left(-\frac{1}{2\varepsilon^2}\int_{s}^{t}r_\varepsilon(s)\right)ds.
\end{align}
We focus for now on the final term in the above equation. We recall $F_2$ may be controlled in terms of the moments as
\begin{align*}
  F_2&\leq{}CA_m(1+|\widetilde{M}_{\varepsilon,1}|+|\widetilde{M}_{\varepsilon,1}|^p+|\widetilde{M}_{\varepsilon,1}|^{p+1})(|M_{\varepsilon,2}|^2+M_{\varepsilon,3}^{|c|}+M_{\varepsilon,4}^{c}+M_{\varepsilon,4+p}^{c}+M_{\varepsilon,2}M_{\varepsilon,2})\\
  &\leq{}CA_m(1+L+L^p+L^{p+1})(K_2^2\varepsilon^4H_{2,\varepsilon}^2+K_2\varepsilon^3H_{2,\varepsilon}+K_2\varepsilon^4H_{2,\varepsilon}+K_2\varepsilon^{4+p}H_{2,\varepsilon}+K_2^2\varepsilon^{4+p}H_{2,\varepsilon}^2)\\
  &\leq{}\varepsilon^3C\tilde{C}_LK_2^2H_{2,\varepsilon}(t)^2
\end{align*}
where $\tilde{C}_L=A_m(1+L+L^p+L^{p+1})$, and we use the fact that $H_{2,\varepsilon}(t)\geq{}1$. We may now integrate to obtain
\begin{align*}
    \left\vert\frac{1}{\varepsilon^2}\int_{0}^{t}F_2(s)\exp\left(-\frac{1}{2\varepsilon^2}\int_{s}^{t}r_\varepsilon(s)\right)ds\right\vert&\leq{}    \frac{1}{\varepsilon^2}\int_{0}^{t}|F_2(s)|\exp\left(-\frac{r_L}{2\varepsilon^2}(t-s)\right)ds\\ &\leq{}C\tilde{C}_LK_2^2\varepsilon\int_{0}^{t}H_{2,\varepsilon}^2(s)\exp\left(-\frac{r_L}{2\varepsilon^2}(t-s)\right)\\
    &\leq{}\varepsilon^{1-\delta}\frac{K_1}{4}\int_{0}^t{}H_{2,\varepsilon}^2(s)\exp\left(-\frac{r_L}{2\varepsilon^2}(t-s)\right)ds,\\
    &\leq{}\varepsilon^{3-\delta}\frac{K_1}{4}\int_{0}^{\frac{r_L}{2\varepsilon 2}t}H_{2,\varepsilon}^2\left(t-\frac{2\varepsilon^2{z}}{r_L}\right)\exp\left(-z\right)dz,
\end{align*}
where the second last inequality is valid for small enough $\varepsilon$ relative to $C\tilde{C}_LK_2^2$.

We have determined
\[|M_{\varepsilon,2}-\varepsilon^2|\leq{}\varepsilon^2(C_1+1)\exp\left(-\frac{r_L}{2\varepsilon^2}t\right)+\varepsilon^{3-\delta}\frac{K_1}{4}\int_{0}^{\frac{r_L}{2\varepsilon 2}t}H_{2,\varepsilon}^2\left(t-\frac{2\varepsilon^2{z}}{r_L}\right)\exp\left(-z\right)dz,\]
which yields \eqref{eqn:M2close} for $K_1$ chosen sufficiently large with respect to $C_1$. To guarantee the right hand side is less than $K_1\varepsilon^2H_{1,\varepsilon}(t)$ we require that $H_{1,\varepsilon}$ satisfies \[H_{1,\varepsilon}\geq{}\exp\left(-\frac{r_L}{2\varepsilon^2}t\right)+\varepsilon^{1-\delta}\frac{1}{4}\int_{0}^{\frac{r_L}{2\varepsilon 2}t}H_{2,\varepsilon}^2\left(t-\frac{2\varepsilon^2{z}}{r_L}\right)\exp\left(-z\right)dz,\]
and we also recall the requirement that \[1+H_{1,\varepsilon}\leq{}H_{2,\varepsilon}.\]
In particular we may choose $H_{1,\varepsilon}$ depending on $H_{2,\varepsilon}$ such that 
\begin{equation}\label{eqn:H1Def}
H_{1,\varepsilon}=\exp\left(-\frac{\eta}{\varepsilon^2}t\right)+\varepsilon^{1-\delta}\frac{1}{4}\int_{0}^{\frac{r_L}{2\varepsilon 2}t}H_{2,\varepsilon}^2\left(t-\frac{2\varepsilon^2{z}}{r_L}\right)\exp\left(-z\right)dz.
\end{equation}
Here we have assumed $\eta<\frac{r_L}{2}$ without loss of generality. This constrains $H_{2,\varepsilon}$ to satisfy 
\begin{equation}\label{eqn:H2Bound1}
    H_{2,\varepsilon}(t)\geq{}1+\exp\left(-\frac{\eta}{\varepsilon^2}t\right)+\varepsilon^{1-\delta}\frac{1}{4}\int_{0}^{\frac{r_L}{2\varepsilon 2}t}H_{2,\varepsilon}^2\left(t-\frac{2\varepsilon^2{z}}{r_L}\right)\exp\left(-z\right)dz.
\end{equation}
We now focus on the higher order moments. 
Because of (H4) we derive from \eqref{eqn:HigherOrderMoments} the inequality
\begin{equation}\label{eqn:HigherOrderMoments2}
\dot{M}_{\varepsilon,2k}^c+\frac{\eta}{\varepsilon^2}{}M_{\varepsilon,2k}^c\leq{}\frac{1}{\varepsilon^2}F_{2k}+|\dot{X}_\varepsilon(t)M_{\varepsilon,2k-1}^{c}|.
\end{equation}
We recall the bound of $F_{2k_0}$ in terms of the moments
\begin{align*}
    F_{2k_0}\leq{}&(|F_1|+CA_m(|\widetilde{M}_{1,\varepsilon}|+|\widetilde{M}_{1,\varepsilon}|^p)M_{\varepsilon,2}^c)|M_{\varepsilon,2k_0-1}^c|\\
    &+CA_m{}M_{\varepsilon,2k_0}^c\left((1+|\widetilde{M}_{\varepsilon,1}|^p)M_{\varepsilon,2}^c+M_{\varepsilon,2+p}^{|c|}\right)\\
    &+\sum_{l=0}^{k_0-1}\sum_{j=0}^{2l}\sigma_{k_0-l}\varepsilon^{2(k_0-l)}\frac{1}{4^l}{2k_0\choose{2l}}{2l\choose{j}}M_{\varepsilon,2l-j}^c{}M_{\varepsilon,j}^c\\
    &+\sum_{j=2}^{2k_0-2}\frac{1}{4^{k_0}}{2k_0\choose{j}}M_{\varepsilon,2k_0-j}^c{}M_{\varepsilon,j}^c,
\end{align*}
%+\widetilde{M}_{\varepsilon,1}^2+\widetilde{M}_{\varepsilon,1}^{p+2} extra term

where $|F_1|\leq{}CA_m((1+|\widetilde{M}_{\varepsilon,1}|^p)M_{\varepsilon,3}^{|c|}+M_{\varepsilon,3+p}^{|c|})$. We first bound \[|F_1||M_{\varepsilon,2k_0-1}^c|\leq{}\varepsilon^{2k_0+2}C\tilde{C}_LK_{2}H_{2,\varepsilon}^2.\]
The next term can be bounded by
\begin{align*}
CA_m(|\widetilde{M}_{1,\varepsilon}|+|\widetilde{M}_{1,\varepsilon}|^p)M_{\varepsilon,2}^c)|M_{\varepsilon,2k_0-1}^c|\leq{}C\tilde{C}_L\varepsilon^{2k_0+1}K_{2}^2H_{2,\varepsilon}^2(t)
\end{align*}
The term on the second line may be bounded by
\begin{align*}
CA_m{}M_{\varepsilon,2k_0}^c\left((1+|\widetilde{M}_{\varepsilon,1}|^p)M_{\varepsilon,2}^c+M_{\varepsilon,2+p}^{|c|}\right)&\leq{}\varepsilon^{2k_0}CA_mH_{2,\varepsilon}((1+L^p)\varepsilon^2{}H_{2,\varepsilon}+\varepsilon^{2+p}H_{2,\varepsilon})\\
&\leq{}\varepsilon^{2k_0+2}C\tilde{C}_LK_{2}^2H_{2,\varepsilon}^2.
\end{align*}
Overall, we have
\begin{align*}
&(|F_1|+CA_m(|\widetilde{M}_{1,\varepsilon}|+|\widetilde{M}_{1,\varepsilon}|^p)M_{\varepsilon,2}^c)|M_{\varepsilon,2k_0-1}^c|\\
&+CA_m{}M_{\varepsilon,2k_0}^c\left((1+|\widetilde{M}_{\varepsilon,1}|^p)M_{\varepsilon,2}^c+M_{\varepsilon,2+p}^{|c|}+\widetilde{M}_{\varepsilon,1}^2+\widetilde{M}_{\varepsilon,1}^{p+2}\right)\leq{}\varepsilon^{2k_0+1}C\tilde{C}_LK_{2}^2H_{2,\varepsilon}^2.
\end{align*}
For the final terms, we estimate
\begin{align*}
M_{\varepsilon,2l-j}^{c}M_{\varepsilon,j}^{c}\leq{}&\varepsilon^{2l}K_{2}^\frac{2k_0-4}{2k_0-2}K_{1}^\frac{4k_0-2l}{2k_0-2}H_{2,\varepsilon}(t)^2\\
\leq{}&\varepsilon^{2l}K_2^\frac{k_0-2}{k_0-1}K_1^\frac{2k_0}{k_0-1}H_{2,\varepsilon}(t)^2,\\
\leq{}&\varepsilon^{2l}\eta_1K_2H_{2,\varepsilon}(t)^2
\end{align*}
where $\eta_1$ may be chosen arbitrarily small if $K_2$ is chosen sufficiently large compared to $K_1$. Altogether, the final line may be bounded as 
\begin{align*}
&\sum_{l=0}^{k_0-1}\sum_{j=0}^{2l}\sigma_{k_0-l}\varepsilon^{2(k_0-l)}\frac{1}{4^l}{2k_0\choose{2l}}{2l\choose{j}}M_{\varepsilon,2l-j}^c{}M_{\varepsilon,j}^c\\
&+\sum_{j=2}^{2k_0-2}\frac{1}{4^{k_0}}{2k_0\choose{j}}M_{\varepsilon,2k_0-j}^c{}M_{\varepsilon,j}^c\leq{}\eta_1C_{k_0}K_2\varepsilon^{2k_0}H_{2,\varepsilon}(t)^2.
\end{align*}
It follows that
\[|F_{2k}|\leq{}\varepsilon^{2k_0}\eta_{2}K_2H_{2,\varepsilon}(t)^2.\]
We may also derive the bound
\[|M^{c}_{2k-1}|\leq{}\varepsilon ^{2k-1}\eta_2{}K_2H_{2,\varepsilon}(t),\]
by again choosing $K_2$ sufficiently large with respect to $K_1$. 
Inserting these controls into \eqref{eqn:HigherOrderMoments2} and integrating yields
\begin{align*}
M_{\varepsilon,2k_0}^c\leq{}&K_2\varepsilon^{2k_0}\exp\left(-\frac{\eta}{\varepsilon^2}t\right)+\varepsilon^{2k_0-2}\eta_2K_{2}\int_{0}^{t}H_{2,\varepsilon}(t)^2\exp\left(-\frac{\eta}{\varepsilon^2}(t-s)\right)ds\\
    &+\varepsilon^{2k_0-1}\eta_2K_2\int_{0}^{t}H_{2,\varepsilon}(s)|\dot{X}_{\varepsilon}(s)|\exp\left(-\frac{\eta}{\varepsilon^2}(t-s)\right)ds.
\end{align*}
we may write the RHS as 
\begin{align*}
\varepsilon^{2k_0}K_2&\left(\exp\left(-\frac{\eta}{\varepsilon^2}t\right)+\varepsilon^{-2}\eta_2\int_{0}^{t}H_{2,\varepsilon}(t)^2\exp\left(-\frac{\eta}{\varepsilon^2}(t-s)\right)ds\right.\\
&\left.+\varepsilon^{-1}\eta_2\int_{0}^{t}H_{2,\varepsilon}(s)|\dot{X}_{\varepsilon}(s)|\exp\left(-\frac{\eta}{\varepsilon^2}(t-s)\right)ds\right).
\end{align*}
From the above, we can now ensure that that inequality \eqref{eqn:M1closeT} holds provided that
\begin{align}\label{eqn:H2Bound2}
H_{2,\varepsilon}(t)\geq{}&\exp\left(-\frac{\eta}{\varepsilon^2}t\right)+\varepsilon^{-2}\eta_2\int_{0}^{t}H_{2,\varepsilon}(t)^2\exp\left(-\frac{\eta}{\varepsilon^2}(t-s)\right)ds\\
&+\varepsilon^{-1}\eta_2\int_{0}^{t}H_{2,\varepsilon}(s)|\dot{X}_{\varepsilon}(s)|\exp\left(-\frac{\eta}{\varepsilon^2}(t-s)\right)ds.\nonumber
\end{align}
We therefore have to find an $H_{2,\varepsilon}(t)$ which satisfies both \eqref{eqn:H2Bound1} and \eqref{eqn:H2Bound2}. Consider $\bar{H}$ which solves
\begin{align}\label{eqn:BarH}
\bar{H}=&1+\exp\left(-\frac{\eta}{\varepsilon^2}t\right)+\varepsilon^{-2}\eta_2\int_{0}^{t}\bar{H}(s)^2\exp\left(-\frac{\eta}{\varepsilon^2}(t-s)\right)ds\\
&+\varepsilon^{-1}\eta_2\int_{0}^{t}\bar{H}(s)|\dot{X}_{\varepsilon}(s)|\exp\left(-\frac{\eta}{\varepsilon^2}(t-s)\right)ds.\nonumber
\end{align}
We now use the Banach fixed-point theorem to show there exists a solution to \eqref{eqn:BarH}, and then show this solution satisfies the required bounds. Let $M_T=\left\{h\in{}C^{1}(\mathbb{R}^{+}): \max_{t\in[0,T]}|h(s)|\leq{3}\right\}$, and $d_T:M\times{M} \mapsto \mathbb{R}^{+}$ such that $d_T(f,g)=\max_{t\in[0,T]}|f(s)-g(s)|$. We see that $(M_T,d_T)$ forms a complete metric space for any given $T$. We next show the mapping $F: M_T \to C^1(\mathbb{R})$ given by
\begin{align*}
F[h]=&1+\exp\left(-\frac{\eta}{\varepsilon^2}t\right)+\varepsilon^{-2}\eta_2\int_{0}^{t}h(s)^2\exp\left(-\frac{\eta}{\varepsilon^2}(t-s)\right)ds\\
&+\varepsilon^{-1}\eta_2\int_{0}^{t}h(s)|\dot{X}_{\varepsilon}(s)|\exp\left(-\frac{\eta}{\varepsilon^2}(t-s)\right)ds.\nonumber
\end{align*}
 is contractive. That is, we show $F(M_T)\subseteq{}M_T$ and for all $h_1,h_2\in{}M_T$ we have that $d_T(F[h_1],F[h_2])<\gamma{}d_T(h_1,h_2)$ for some $\gamma\in(0,1)$. We obtain for $h\in{}M_T$, using (H7) for the second integral,
 \[F[h]\leq{}2+9\frac{\eta_2}{\eta}+3\eta_2L_X.\]
 This shows $F(M_T)\subseteq{}M_T$ provided $\eta_2\leq{}\frac{\eta}{9+3L_X\eta}$, and so we impose this condition on $\eta_2$. To show the second part
 we write, for $h_1,h_2\in{}M_T$,
\begin{align*}
    |F[h_1](t)-F[h_2](t)|=&\left\vert\varepsilon^{-2}\eta_2\int_{0}^{t}(h_1(s)^2-h_2(s)^2)\exp\left(-\frac{\eta}{\varepsilon^2}(t-s)\right)\right.\\
&\left.+\varepsilon^{-1}\eta_2\int_{0}^{t}(h_1(s)-h_2(s))|\dot{X}_{\varepsilon}(s)|\exp\left(-\frac{\eta}{\varepsilon^2}(t-s)\right)ds\right\vert,\\
    &\leq{}\eta_2\left(\frac{6}{\eta}+L_X\right)d_T(h_1,h_2),\\
    &\leq{}\frac{\eta_2(6+L_X\eta)}{\eta}d_T(h_1,h_2),
\end{align*}
where we again make use of (H7) in the second last line. We have that $\frac{\eta_2(6+L_X\eta)}{\eta}<1$ from $\eta_2\leq{}\frac{\eta}{9+3L_X\eta}$ . This confirms that $F$ is a contractive mapping. Applying the Banach-fixed point theorem then ensures there exists a unique solution $\bar{H}\in{}M_T$.

We choose $H_{2,\varepsilon}(t)=\bar{H}$. The bound \eqref{eqn:H2Bound2} is satisfied immediately. Furthermore, $\bar{H}(0)=2$, so $H_{2,\varepsilon}(0)=2$, as we assumed, and by the the choice of $M_T$, we have that $\bar{H}(t)\leq{3}$  for $t\geq{0}$. The bound  $H_{1,\varepsilon}(t)\leq{}3\varepsilon^{1-\delta}$ for $t$ sufficiently large follows from using the bound $H_{2,\varepsilon}(t)<3$ in \eqref{eqn:H1Def}. It remains to check the choice $H_{2,\varepsilon}(t)=\bar{H}(t)$ satisfies \eqref{eqn:H2Bound1}, but this is straightforward for $\varepsilon<1$.

% We let $I_1(t)=\int_{0}^{t}\bar{H}^2(s)\exp\left(-\frac{\eta}{\varepsilon^2}(t-s)\right)ds$ and $I_2(t)=\int_{0}^{t}\bar{H}(s)|\dot{X}_{\varepsilon}(s)|\exp\left(-\frac{\eta}{\varepsilon^2}(t-s)\right)ds$, so that
% \[\bar{H}=\exp\left(-\frac{\eta}{\varepsilon^2}t\right)+\varepsilon^{-2}\eta_2I_1+\varepsilon^{-1}I_2.\]

% Differentiating the above yields
% \begin{align*}
%     \dot{\bar{H}}=&-\eta\varepsilon^{-2}\exp\left(-\frac{\eta}{\varepsilon^2}t\right)+\varepsilon^{-2}\eta_2K_2\left(-\varepsilon^{-2}\eta{}I_1+\bar{H}^2(t)\right)\\
%     &+K_2\varepsilon^{-1}(-\varepsilon^{-2}\eta{}I_2+\bar{H}|\dot{X}_\varepsilon|)\\
%     =&-\eta\varepsilon^{-2}\bar{H}+\varepsilon^{-2}\eta_2\bar{H}^2+\varepsilon^{-1}\bar{H}|\dot{X}_\varepsilon|\\
%     =&-\varepsilon^{-2}\left(\eta-\varepsilon{}|\dot{X}_\varepsilon|-\eta_2\bar{H}\right)\bar{H}.
% \end{align*}

% We now do a similar procedure: we let $T_{2,\varepsilon}$ be the time interval for which $\bar{H}\leq{}C_H$. On this interval we have that
% \[\dot{\bar{H}}\leq{}-\varepsilon^{-2}(\eta-\eta_2C_H)\bar{H}+\varepsilon^{-1}|\dot{X}_\varepsilon|\]
% from which we may estimate 
% \[\bar{H}\leq{}\varepsilon^{-1}\int_{0}^{t}\exp\left(-\frac{\eta-\eta_2C_H}{\varepsilon^2}(t-s)\right)|\dot{X}_\varepsilon(s)|ds+\exp\left(-\frac{\eta-\eta_2C_H}{\varepsilon^2}t\right).\]
% We can then choose $\eta_2$ small enough and $C_H$ large enough so $\bar{H}<C_H$ over the interval $(0,T_{2,\varepsilon})$ which implies we can extend the interval to $(0,T_\varepsilon)$ if necessary, independent of $T_\varepsilon$. 

We now proceed to write the evolution equation for $\widetilde{M}_{\varepsilon,1}$ as
    \[\dot{\widetilde{M}}_{\varepsilon,1}+\dot{X}_\varepsilon(t)+m_{t}'(\widetilde{M}_{\varepsilon,1})=-\tilde{F}_1,\]
    where \begin{align*}
        |\tilde{F}_1|&\leq{}\frac{1}{\varepsilon^ 2}|F_1|+\frac{1}{\varepsilon^ 2}|m_t'(\widetilde{M}_{\varepsilon,1})||M_{\varepsilon,2}^c-\varepsilon^2|\\
        &\leq{}C\tilde{C}_LK_1H_{1,\varepsilon}(t).
    \end{align*}
Here we have used \eqref{eqn:F1bound} and the fact that $|M_{\varepsilon,1}^c-X_\varepsilon(t)|<L$ for $t\in[0,T_\varepsilon]$.

We now estimate $N_1=\tilde{m}_\varepsilon-Y_\varepsilon$  which satisfies
\[\dot{N}_1+m_{t}''(Y_\varepsilon)N_1=\tilde{F}_1-(m_t'(M_{\varepsilon,1})-m_t'(Y_\varepsilon)-m_{t}''(Y_\varepsilon)N_1).\]
We may use finite difference theorems to bound $|(m_t'(M_{\varepsilon,1})-m_t'(Y_\varepsilon)-m_{t}''(Y_\varepsilon)N_1)|<C_mK_0\varepsilon^{1-\delta}|N_1|$. Then
combining the convexity from (H1), the bound on the lag $|Y_\varepsilon|<L$ , and the $t$-independent bound of $\partial_{xxx}m(x,t)$ from (H3), we obtain 
\[\frac{d}{dt}\left[\frac{|N_1|^2}{2}\right]+A_0|N_1|^2\leq{}(\tilde{F}_1+C_mK_0\varepsilon^{1-\delta}|N_1|)N_1.\]
We arrive at
\[\frac{d}{dt}\left[\frac{|N_1|^2}{2}\right]+(A_0-C_{m}K_0\varepsilon^{1-\delta})|N_1|^2\leq{}\tilde{F}_1|N_1|.\]
As in \cite{GHM}, we estimate $\tilde{F}_1|N_1|\leq{}\frac{\tilde{F}_1^2}{2A_0}+\frac{A_0|N_1|^2}{2}$ and apply Gronwall's inequality yielding
\begin{align*}
    \frac{|N_1|^2}{2}\leq{}&\frac{C}{A_0}\int_{0}^{t}|\tilde{F}_1(s)|^2\exp\left({-A_0(t-s)}\right)\\
    \leq{}&\frac{C^2\tilde{C}_L^2K_1^2}{A_0}\int_{0}^{t}H_{1,\varepsilon}(s){e^{-A_0(t-s)}}ds\\
    \leq{}&\frac{C^2\tilde{C}_L^2K_1^2}{A_0}\left(\frac{\varepsilon^{2(1-\delta)}}{A_0}+\frac{\varepsilon^2}{r_L}\right).
\end{align*}

This shows the final inequality for large enough $K_0$ depending on $K_1$ and $L$.
 \end{proof}

Proposition \ref{prop:InfiniteT_eps} follows, since we have shown the inequalities (\ref{eqn:M2kcloseT})-(\ref{eqn:M1closeT}) are strict at $T_\varepsilon$. This is only possible if $T_\varepsilon=\infty$ since otherwise we could have taken a large $T_\varepsilon$, but this would contradict its definition.

\subsection{Proof of \cref{thm:main}} Having shown the moment estimates of \cref{thm:MomentEstimates} for the time-dependent mortality function,
\cref{thm:main} now follows from analogous computations as in \cite{GHM}, and we refer the reader there for details. Here we sketch only the main steps. The key distributions are: the trait-distribution $q_\varepsilon(x,t)$, the  Gaussian $g_\varepsilon(x,t)$ defined in \eqref{eqn:approximate_distribution}, with mean given by  $\bar{Z}_\varepsilon(t)$ that solves \eqref{eqn:approx_mean_trait}, and the intermediate Gaussian $\bar{g}_\varepsilon(x,t)$ centred on the  mean trait ${M}_{\varepsilon,1}(t)$, which we define below
\begin{equation}\label{eqn:approximate_distribution2}
    \bar{g}_\varepsilon(x,t)=\frac{1}{\sqrt{2\pi\varepsilon^2}}\exp\left(-\frac{(x-{M}_{\varepsilon,1}(t))^2}{2\varepsilon^2}\right).
\end{equation}

We compute a duality estimate for $q_\varepsilon-g_\varepsilon$ via 
\[I_\phi(t)=\int_{\mathbb{R}}\phi(x)(q_\varepsilon-g_\varepsilon).\]
Using \eqref{eq: qeps}, we determine
\[\varepsilon^2{}\dot{I}_\phi=rT_{r}-rI_\phi-T_s,\]
where 
\[T_r=\int_{\mathbb{R}}\phi(x)(\tilde{T}_\varepsilon[q_\varepsilon]-\tilde{T}_\varepsilon[g_\varepsilon]),\]
\begin{align*}
    T_s=&\int_{\mathbb{R}}\tilde{m}_\varepsilon(q_\varepsilon-g_\varepsilon)-\left(\left(\int_{\mathbb{R}}\widetilde{M}q_\varepsilon\right)\left(\int_{\mathbb{R}}q_\varepsilon\phi\right)-\widetilde{M}(\bar{Z}_\varepsilon)\left(\int_{\mathbb{R}}g_\varepsilon\phi\right)\right)\\
    &-\int_{\mathbb{R}}(x-\bar{Z}_\varepsilon)^2r^{\widetilde{M}}[\bar{Z}_\varepsilon](x)g_\varepsilon(x,t)\phi(x)dx\\
    &=T_s^{(a)}-T_s^{(b)}+T_s^{(c)}.
\end{align*}
In all cases, terms are controlled by obtaining bounds in terms of the difference $|M_{\varepsilon,1}-\bar{Z}_\varepsilon|$ which has been shown to be $O(\varepsilon^{1-\delta})$ according to Proposition \ref{prop:InfiniteT_eps}. The first term $T_r$ is controlled by writing it as 
\[\int_{\mathbb{R}}\phi(\tilde{T}_\varepsilon[q_\varepsilon]-\tilde{T}_\varepsilon[\bar{g}_\varepsilon])+\int_{\mathbb{R}}\phi(\tilde{T}_\varepsilon[\bar{g}_\varepsilon]-\tilde{T}_\varepsilon[g_\varepsilon]),\]
then using the following contraction property of the operator $\tilde{T}_\varepsilon$, derived in \cite[Appendix 4.2]{raoul2017macroscopic}, to bound the first part. We have for $g_1,g_2\in{}L^2(\mathbb{R})$ that $W_2(\tilde{T}(g_1),\tilde{T}(g_2))\leq{}\frac{1}{\sqrt{2}}W_2(g_1,g_2),$
leading to \[\int_{\mathbb{R}}\phi(\tilde{T}_\varepsilon[q_\varepsilon]-\tilde{T}_\varepsilon[\bar{g}_\varepsilon])\leq{}\frac{\Vert{\phi'}\Vert_{L^\infty}}{\sqrt{2}}\left(W_2(q_\varepsilon,g_\varepsilon)+W_2(g_\varepsilon,\bar{g}_\varepsilon)\right)\]
The next part is more straightforward to control using properties of $W_1$ and $\tilde{T}_\varepsilon$.

The term $T_s^{(c)}$ is controlled by using (H3) to bound $r^{\widetilde{M}}[\bar{Z}_\varepsilon]$, as well as the bound $|{M}_\varepsilon-X_\varepsilon(t)|<L$. The remaining terms are controlled by Taylor expanding $\widetilde{M}(x,t)$ about $M_{\varepsilon,1}$ or $\bar{Z}_\varepsilon$ as appropriate. This leads to 
\[\varepsilon^{2}\dot{I}_\phi+rI_\phi=F,\]
where \[|F|\leq{}\Vert\phi'\Vert_{L^\infty}\left(rc_0W_1(q_\varepsilon,g_\varepsilon)+\frac{K}{c_0^{\frac{1}{2}}}\varepsilon^{1-\delta}\right),\]
and $c_0>0$ is arbitrarily small.
Theorem \ref{thm:main} then follows from the Gronwall inequality, a sup argument on $\phi$ and $t\in[0,T]$ for fixed $T<\infty$, and the arbitrariness of $c_0$.

\section{Proof of \cref{thm:IP_main}}
\label{sec: proof of 2nd thm}
In this section, we prove \cref{thm:IP_main} that quantitatively establishes the asymptotic limits of the trait distribution and total population size for the prey-predator model \eqref{eqn:evolving_prey_simplified}. We will need the following technical lemma, which will allow us to bypasses the need for (H2).

\begin{lemma}\label{lma:KPPBound}

Suppose that $y_\varepsilon(t)$ solves 
\begin{equation}\label{eqn:GenKPP}
    \begin{cases}
\varepsilon^2\frac{dy_\varepsilon}{dt}=\alpha_\varepsilon(t)(H_\varepsilon(t)-y_\varepsilon)y_\varepsilon,\\
y_\varepsilon(0)=H_\varepsilon(0)+O(\varepsilon),
\end{cases}
\end{equation}
where there exit constants $\alpha_L,\alpha_U,H_L,H_U$ such that $0<\alpha_L<\alpha_\varepsilon(t)<\alpha_U$, $0<H_L<H_\varepsilon(t)<H_U$, $H_\varepsilon'(t)<C_H$, and $\alpha'_\varepsilon(t)<C_\alpha$. Then 
\begin{itemize}
    \item[(a)] $|y_\varepsilon(t)-H_\varepsilon(t)|\leq{}C\varepsilon$, where $C$ is a constant depending on $C_H$, $H_L$, $H_U$ and initial condition.
    \item[(b)] $\dot{y}_\varepsilon(t)\leq{}C_1+\varepsilon^{-1}C_2\exp^{-\frac{H_L}{2\varepsilon^2}t}.$
\end{itemize}
\end{lemma}

\begin{proof}[Proof of Lemma \ref{lma:KPPBound}]
Letting $J_\varepsilon(t):=\frac{1}{y_\varepsilon(t)}$ we find that 
\[J_\varepsilon(t)=J_\varepsilon(0)e^{-\frac{\int_{0}^{t}\alpha_\varepsilon(s)H_\varepsilon(s)ds}{\varepsilon^2}}+\varepsilon^{-2}\int_{0}^{t}\alpha_{\varepsilon}(s)e^{-\frac{1}{\varepsilon^2}\int_{s}^{t}H_\varepsilon(\tau)\alpha_\varepsilon(\tau)d\tau}ds.\]
We focus for now on the latter term. Let $\beta\in(1,2)$, and suppose $t<\varepsilon^\beta$. Then for $0<s<t$ we have
\[|H_\varepsilon(s)-H_\varepsilon(t)|\leq{}C_H{}\varepsilon^\beta,\]
 due to the $\varepsilon$ independent bounds on $H_\varepsilon(t)$. We obtain, using also the uniform lower bounds on $\alpha_\varepsilon(t)$ and $H_\varepsilon(t)$,
\[J_\varepsilon(t)\leq{}J_\varepsilon
(0)e^{-\frac{H_\varepsilon(t)-C_H\varepsilon^\beta}{\varepsilon^2}\int_{0}^{t}\alpha_\varepsilon(\tau)d\tau}+\varepsilon^{-2}\int_{0}^{t}\alpha_\varepsilon(s)e^{-\frac{H_\varepsilon(t)-C_H\varepsilon^\beta}{\varepsilon^2}\int_{s}^{t}\alpha_\varepsilon(\tau)d\tau}ds.\]
A change of variables allows us to simplify the RHS to give 
\[J_\varepsilon(t)\leq{}J_\varepsilon
(0)e^{-\frac{H_\varepsilon(t)-C_H\varepsilon^\beta}{\varepsilon^2}\int_{0}^{t}\alpha_\varepsilon(\tau)d\tau}+\frac{1}{H_\varepsilon(t)-C_H\varepsilon^\beta}\left(1-e^{-\frac{H_\varepsilon(t)-C_H\varepsilon^\beta}{\varepsilon^2}\int_{0}^{t}\alpha_\varepsilon(\tau)d\tau}\right).\]
A similar set of computations provides a lower bound
\[J_\varepsilon(t)\geq{}J_\varepsilon
(0)e^{-\frac{H_\varepsilon(t)+C_H\varepsilon^\beta}{\varepsilon^2}\int_{0}^{t}\alpha_\varepsilon(\tau)d\tau}+\frac{1}{H_\varepsilon(t)+C_H\varepsilon^\beta}\left(1-e^{-\frac{H_\varepsilon(t)+C_H\varepsilon^\beta}{\varepsilon^2}\int_{0}^{t}\alpha_\varepsilon(\tau)d\tau}\right).\]
From this we obtain, for $t\in(0,\varepsilon^\beta)$, \[J_\varepsilon\in\left(\min\left\{J_\varepsilon(0),\frac{1}{H_\varepsilon(t)+C_H\varepsilon^\beta}\right\},\max\left\{J_\varepsilon(0),\frac{1}{H_\varepsilon(t)-C_H\varepsilon^\beta}\right\}\right).\]
We now use the fact that $y_\varepsilon(0)=H_\varepsilon(0)+O(\varepsilon)=H_\varepsilon(t)+O(\varepsilon)$, to deduce that $y_\varepsilon(t)=H_\varepsilon(t)+O(\varepsilon)$ for $t\in[0,\varepsilon^\beta].$

If $t>\varepsilon^\beta$ instead, we find that the first term is exponentially small, and can estimate the latter similarly to above. Overall this shows $y_\varepsilon(t)=H_\varepsilon(t)+O(\varepsilon)$. This completes part $(a)$. 

We now show part $(b)$. First, we differentiate \eqref{eqn:GenKPP} and set $\nu_\varepsilon=\dot{y}_\varepsilon$.  We obtain 
\begin{align*}
\varepsilon^{2}\dot{\nu}_\varepsilon&=(H_\varepsilon(t)-y_\varepsilon)\nu_\varepsilon-(H_\varepsilon'(t)-\nu_\varepsilon)y_\varepsilon\\
&=-H_\varepsilon'(t)y_\varepsilon-(2y_\varepsilon-H_\varepsilon(t))\nu_\varepsilon.
\end{align*}
Using part (a), we find that $2y_\varepsilon-H_\varepsilon=H_\varepsilon(t)+O(\varepsilon)$ and so
\[\varepsilon^{2}\dot{v}_\varepsilon\leq{}2C_HH_U-\frac{H_L}{2}\nu_\varepsilon,\]
when $\nu_\varepsilon>0$ and
\[\varepsilon^{2}\dot{v}_\varepsilon\geq{}-2C_HH_U-\frac{H_L}{2}\nu_\varepsilon,\]
when $\nu_\varepsilon<0$. We therefore have, for $\overline{\nu}_\varepsilon=|\nu_\varepsilon|$,
\[\varepsilon^2\overline{\nu}_\varepsilon\leq{}2C_HH_U-\frac{H_L}{2}\overline{\nu}_\varepsilon,\]
from which we derive
\[
|\nu_\varepsilon|\leq{}2C_H\rho_U\left(1-e^{-\frac{H_L}{2\varepsilon^2}t}\right)+|\nu_\varepsilon(0)|e^{-\frac{H_L}{2\varepsilon^2}t}.
\]
Altogether, this gives $|\nu_\varepsilon|\leq{}C_1+C_2\varepsilon^{-1}e^{-\frac{H_L}{2\varepsilon^2}t}$, proving part (b). 

\end{proof}

\begin{proof}[{Proof of \cref{thm:IP_main}}]
The proof consists of verifying each hypothesis, which we first do locally in time, and then show this implies the conditions hold for all $t>0$ since relevant inequalities remain strict.  We begin with (H1). By definition $\partial_{x}m_{\varepsilon}(0,t)=0$. At $t=0$, we also have that $\partial_{xx}m(0,t)>A_0$ due to assumptions (A3) and (A2). By continuity, we have that this inequality remains true for $t\in[0,T_\varepsilon]$ for some $T_\varepsilon>0.$ In particular, we show this inequality remains strict at $t=T_\varepsilon$ by the end of Step 2.

%Similarly, (H9) holds by continuity on some interval $[0,T_\varepsilon]$.  

To see that (H3) is satisfied notice that (A4) ensures that there is a constant ${A}_m$ such that $|\partial_{xxx}{m}_{\varepsilon}(x,t)|<A_m\rho_{1,\varepsilon}(t)(1+|x+X_\varepsilon|^{p-1})$. We note that it is straightforward to bound $\rho_{1,\varepsilon}<\frac{r_1}{\kappa^{*}}$ under assumption (A0), where we follow similar steps as in \cref{lma:KPPBound}. Therefore it follows that $|\partial_{xxx}{m}_{\varepsilon}(x,t)|<A_m\rho_{1,\varepsilon}(t)(1+|x+X_\varepsilon|^{p-1})\leq{}\frac{r_1}{\kappa_1^{*}}A_m(1+|x+X_\varepsilon|^{p-1})$. We will shortly show that $X_\varepsilon(t)$ is bounded as a consequence of $|\bar{Z}_\varepsilon-x^{*}|<L_1$ and estimates on the moments valid for some initial time.  Assumption (H4) corresponds exactly to (A5), and (H5) is contained within (A6). The expressions for the upper and lower bounds of $\rho_{1,\varepsilon}(0)$, as required from (H6), come from (A3); we have 
$-K_3\varepsilon+I(x_0)<\rho_{1,\varepsilon}(0)<K_3\varepsilon+I(x_0)$.

So far the assumptions have been satisfied trivially (at least locally), but (H7) requires a more delicate approach. From assumption (A5) and continuity of the solution, the following hold on an interval $[0,T_\varepsilon]$ for some $T_\varepsilon>0$, and large enough constants $K_0$, $K_1$, and $K_2$:
\begin{align}
 \label{eqn:M2kcloseIP} M^{c}_{\varepsilon,2k_0}(t) &  <  K_{2} \varepsilon^{2k_0}H_{2,\varepsilon}(t), \\
\label{eqn:M2closeIP} 
 \left | M_{\varepsilon,2}^{c}(t) -  \varepsilon^2  \right| & <  K_1\varepsilon^2 H_{1,\varepsilon}(t), \\
\label{eqn:M1closeIP} \left| M_{\varepsilon,1}(t) - \bar{Z}_{\varepsilon}(t) \right|  & <  K_0 \varepsilon^{1-{\delta}},
\end{align}
where $H_{1,\varepsilon}(t)$ and $H_{2,\varepsilon}(t)$ are defined in terms of $X_\varepsilon(t)$ as in the proof of \cref{prop:InfiniteT_eps}.

 We firstly let $0<T_2<T_\varepsilon$ be such that $|\bar{Z}_\varepsilon-x^{*}|<L_1$ for $t\in[0,T_2]$. We will show that $T_2=T_\varepsilon$ by the end of Step 2. 
 
\textbf{Step 1.} Here we obtain estimates on the averaged functions $\bar{f}[q_{1,\varepsilon}]$ and $\bar{\delta}[q_{1,\varepsilon}]$. 
A Taylor expansion of $\bar{\delta}[q_{1,\varepsilon}](t)$ about $M_{\varepsilon,1}$ gives
\[\bar{\delta}[q_{1,\varepsilon}]=\delta(M_{\varepsilon,1})+\int_{\mathbb{R}}(x-M_{\varepsilon,1})^2r^\delta[M_{\varepsilon,1}(t)](x)q_{1,\varepsilon}(t),\]
where we recall $r^\delta[M_{\varepsilon,1}](x)=\int_{0}^{1}(1-\sigma)\partial_{xx}\delta(M_{\varepsilon,1}+\sigma(x-M_{\varepsilon,1}))d\sigma$. Using (A4) we may bound this remainder uniformly as
$|r^\delta[M_{\varepsilon,1}](x)|\leq{}CA_m(1+|M_{\varepsilon,1}|^p+|x-M_{\varepsilon,1}|^p).$
Similar computations  may be applied to $\bar{f}[q_{1,\varepsilon}]$.
From this bound we write
\begin{align*}
    |\bar{\delta}[q_{1,\varepsilon}]-\delta(M_{\varepsilon,1})|<CA_m(M_{\varepsilon,2}(1+|M_{\varepsilon,1}|^p)+M_{\varepsilon,p+2}),\\
      |\bar{f}[q_{1,\varepsilon}]-f(M_{\varepsilon,1})|<CA_m(M_{\varepsilon,2}(1+|M_{\varepsilon,1}|^p)+M_{\varepsilon,p+2}). 
\end{align*}
Combining with \eqref{eqn:M1closeIP}, we determine that $\bar{\delta}[q_{1,\varepsilon}]=\delta(\bar{Z}_{\varepsilon})+O(|\bar{Z}_\varepsilon|^p\varepsilon^{1-\delta})$, and $\bar{f}[q_{1,\varepsilon}]=f(\bar{Z}_\varepsilon)+O(|\bar{Z}_\varepsilon|^p\varepsilon^{1-\delta})$ on $t\in[0,T_\varepsilon]$, and  $\bar{\delta}[q_{1,\varepsilon}]=\delta(\bar{Z}_{\varepsilon})+O(\varepsilon^{1-\delta})$, and $\bar{f}[q_{1,\varepsilon}]=f(\bar{Z}_\varepsilon)+O(\varepsilon^{1-\delta})$ on $t\in[0,T_2]$.

\textbf{Step 2.} We now use the estimates from Step 1 above to  bound $\rho_\varepsilon$ and $|\bar{Z}_\varepsilon|$. We may write the equation \eqref{eqn:IP_population_size} as
\begin{align*}
\frac{d\rho_{1,\varepsilon}}{dt}&=\varepsilon^{-2}\left(r_1-\rho_{1,\varepsilon}\left(\left(\frac{\bar{f}[q_{1,\varepsilon}](t)}{1+h\bar{f}[q_{1,\varepsilon}](t)}\rho_{1,\varepsilon}\right)^2+\kappa_1-\bar{\delta}[q_{1,\varepsilon}.](t)\right)\right)\rho_{1,\varepsilon}
\end{align*}

We recall the functions  $F:(x,I,C)\in\mathbb{R}^3\mapsto{}\mathbb{R}$ defined by $F(x,I,C)=\frac{f(x)C}{(1+hCI)^2}$, and $G(x,I): (x,I)\in\mathbb{R} ^2\mapsto\mathbb{R}$ defined by $G(x,I)=r_1-(F(x,I,f(x))^2+\kappa_1-\delta(x))I$. Using Step 1, we have on $[0,T_2]$
\begin{equation*}
\frac{d\rho_{1,\varepsilon}}{dt}=\varepsilon^{-2}\left(r_1-\rho_{1,\varepsilon}(F(\bar{Z}_\varepsilon,\rho_{1,\varepsilon},f(\bar{Z}_\varepsilon))^2+\kappa_1-\delta(\bar{Z}_\varepsilon)+O(\varepsilon^{1-\delta}))\right)\rho_{1,\varepsilon}.  
\end{equation*}
 We may therefore write
\begin{equation}\label{eqn:pop_dymn_app1}
\frac{d\rho_{1,\varepsilon}}{dt}=\varepsilon^{-2}(G(\bar{Z}_\varepsilon,\rho_{1,\varepsilon})+O(\varepsilon^{1-\delta}))\rho_{1,\varepsilon}.
\end{equation}
The function $(1+f(x)hI)^2G(x,I)$ is a cubic in $I$ with unique positive root $I(x)$ satisfying $I_L<I(x)<I_U$, provided $|x-x^{*}|<L_1$ (following from assumption (A1)). We therefore have the identity
\[G(x,I)(1+f(x)hI)^2=P(x,I)(I(x)-I),\] where $P(x,I)$ is a quadratic in $I$ whose leading term has a positive coefficient. Taking the partial derivative with respect to $I$, and using (A1), yields
$P(x,I(x))>G^{*}(1+f(x)hI(x))^2$ which in particular gives $P(x,I)>\frac{G^{*}}{2}$ provided $|I-I(x)|<L_2$ for some $L_2>0$. Applying the above identity in (\ref{eqn:pop_dymn_app1}) and rearranging yields
\begin{equation}\label{eqn:pop_dymn_app12}
\frac{d\rho_{1,\varepsilon}}{dt}=\frac{P(\bar{Z}_\varepsilon,\rho_{1,\varepsilon})}{(1+f(\bar{Z}_\varepsilon)h\rho_{1,\varepsilon})^2}\varepsilon^{-2}(I(\bar{Z}_\varepsilon)-\rho_{1,\varepsilon}+O(\varepsilon^{1-\delta}))\rho_{1,\varepsilon}.
\end{equation}
from which we may derive that $\rho_{1,\varepsilon}=I(\bar{Z}_\varepsilon)+O(\varepsilon^{1-\delta})$ for $t\in[0,T_2]$. 
This follows from \cref{lma:KPPBound}, the bound on $\rho_{1,\varepsilon}(0)$ from (A3), and the fact that the first derivative of $I(\bar{Z}_\varepsilon)$ is bounded independently of $\varepsilon$. The last fact also relies on the assumption that $I(x)$ changes continuously in $x$ for $|x-x^{*}|<L_1$. As a consequence, $P(\bar{Z}_\varepsilon,\rho_{1,\varepsilon})=P(\bar{Z}_\varepsilon,I(\bar{Z}_\varepsilon))+O(\varepsilon^{1-\delta})>\frac{G^{*}}{2}$ for $t\in[0,T_2]$ for $\varepsilon$ small enough depending on $L_2$.

We similarly show the dynamics of $\bar{Z}_\varepsilon$ for $t\in[0,T_2]$ are given by
\begin{equation}\label{eq:trait_dynamics_approxim}
    \dot{\bar{Z}}_\varepsilon=-(\partial_{x}F(\bar{Z}_\varepsilon,I(\bar{Z}_\varepsilon))F(\bar{Z}_\varepsilon,I(\bar{Z}_\varepsilon))-\partial_{x}\delta(\bar{Z}_\varepsilon))I(\bar{Z}_\varepsilon)+O(\varepsilon^{1-\delta}).
\end{equation}
Assumption (A2)  yields that
\begin{equation}\label{eq:trait_dynamics_approxim}
    \dot{\bar{Z}}_\varepsilon=-A_0|x^{*}+O(\varepsilon^{1-\delta})-\bar{Z}_\varepsilon|.
\end{equation}
which implies $\bar{Z}_\varepsilon$ converges to within $O(\varepsilon^{1-\delta})$ of $x^{*}$ and in particular at $|\bar{Z}_\varepsilon(T_2)-x^{*}|<L_1$ for small enough $\varepsilon$. This shows that $T_2=T_\varepsilon$.

As a consequence, we also have for $t\in[0,T_\varepsilon]$ the following: $\partial_{xx}m(0,t)>A_0$ as required for (H1) (i.e. the strict inequality holds to the end of the interval), and $X_\varepsilon$ is uniformly bounded using (A2) and the bounds on $\bar{Z}_\varepsilon$, thereby showing (H3) is satisfied at the end point of $[0,T_\varepsilon]$.

\textbf{Step 3.} We now use (5.1)-(5.3) to control the derivative of $\bar{f}[q_{1,\varepsilon}](t)$, since this is  needed to control  $\dot{\rho}_{\varepsilon,1}(t)$ and ultimately $\dot{X}_\varepsilon(t)$.
The evolution of $q_{1,\varepsilon}(t)$ is governed by 
\begin{equation}\label{eqn:distribution_IP}
\partial_{t}q_{1,\varepsilon}= r_1 \left(\widetilde T_{\varepsilon}[q_{1,\varepsilon}] - q_{1,\varepsilon} \right) 
-  \left( \widetilde{M}_{\varepsilon}(x,t)  - \int_{\mathbb R} m(x-X_\varepsilon(t),t)q_{1,\varepsilon}(x,t) dx \right) q_{1,\varepsilon}.
\end{equation}
By multiplying \eqref{eqn:distribution_IP} by $f$, and integrating we obtain the following equation for the evolution of $\bar{F}$:
\begin{equation}\label{eqn:barf_evo}
\frac{\bar{f}[q_{1,\varepsilon}]}{dt}=\varepsilon^{-2}\left(r_1(\bar{f}[\tilde{T}_{\varepsilon}[q_{1,\varepsilon}]](t)-\bar{f}[q_{1,\varepsilon}])-\overline{mf}[q_{1,\varepsilon}]+\bar{m}[q_{1,\varepsilon}]\bar{f}[q_{1,\varepsilon}]\right) 
\end{equation}
By similar computations as in Step 1 we find  $|\overline{mf}[q_{1,\varepsilon}]-\bar{m}[q_{1,\varepsilon}]\bar{f}[q_{1,\varepsilon}]|<\frac{C_{f}}{2}\varepsilon^2$ for some constant $C_f$ depending on $f$, $\max_{t\in[0,T_\varepsilon]}|{X}_\varepsilon(t)|$ and $L_1$. We also require the bounds obtained in Step 1 here since $\bar{m}[q_{1,\varepsilon}]$ and $\overline{mf}[q_{1,\varepsilon}]$ depend on $\rho_{1,\varepsilon}(t)$. An identical estimate applies to $r_1(\bar{f}[\tilde{T}_{\varepsilon}[q_{1,\varepsilon}]](t)-\bar{f}[q_{1,\varepsilon}])$, potentially enlarging $C_{f}$, and so we finally obtain 
\[\frac{\bar{f}[q_{1,\varepsilon}]}{dt}<C_f,~~t\in[0,T_\varepsilon].\]
We may similarly bound 
\[\frac{\bar{\delta}[q_{1,\varepsilon}]}{dt}<C_\delta,~~t\in[0,T_\varepsilon].\]
\textbf{Step 4.} The goal is now to control $\vert{\dot{X}_\varepsilon(t)}|$. We begin by bounding $\frac{d\rho_{1,\varepsilon}}{dt}$ which is the essential difficulty.
By following similar computations as in Step 2, we express (\ref{eqn:IP_population_size}) as
\begin{align*}
\frac{d\rho_{1,\varepsilon}}{dt}&=\frac{P(M_{\varepsilon,1},\rho_{1,\varepsilon})}{(1+f(M_{\varepsilon,1})h\rho_{1,\varepsilon})^2}\varepsilon^{-2}(I(M_{\varepsilon,1})-\rho_{1,\varepsilon}+O(\varepsilon))\rho_{1,\varepsilon}.
\end{align*}

We now verify the assumptions of \cref{lma:KPPBound}. By assumption (A3) we have that ${\rho}_{1,\varepsilon}(0)=I(M_{1,\varepsilon}(0))+O(\varepsilon)$. We also can see that  $\frac{dI(M_{\varepsilon,1})}{dt}=\dot{M}_{\varepsilon,1}I'(M_{\varepsilon,1})$ which is bounded by a constant depending on $\partial_{x}\widetilde{M}(x,t)$ and $\sup_{|x-x^{*}|<L_1}|I'(x)|$ and so $\left|\frac{d\rho_{1,\varepsilon}}{dt}(t)\right|<K_4\varepsilon^{-1}$ for $t\in[0,T_\varepsilon]$ follows from \cref{lma:KPPBound} for a constant $K_4$ independent of $\varepsilon$.

The function ${X}_\varepsilon(t)$ is implicitly defined as a positive minimum (in $x$) of $F(x,\rho_{1,\varepsilon},\bar{f}[q_{1,\varepsilon}](t))-\delta(x)$.
Therefore $X_\varepsilon(t)$ satisfies 
$\partial_{x}F(X_\varepsilon(t)+x_{m,0,\varepsilon},\rho_{1,\varepsilon},\bar{f}[q_{1,\varepsilon}](t))-\partial_{x}\delta(X_\varepsilon(t)+x_{m,0,\varepsilon})=0.$ Differentiating with respect to time yields
\begin{align*}
&\dot{X}_\varepsilon(t)(\partial_{xx}F(X_\varepsilon+x_{m,0,\varepsilon},\rho_{1,\varepsilon},\bar{f}[q_{1,\varepsilon}](t)))-\partial_{xx}\delta(X_\varepsilon+x_{m,0,\varepsilon}))\\
&-\partial_{xI}F(X_\varepsilon(t)+x_{m,0,\varepsilon},\rho_{1,\varepsilon},\bar{f}[q_{1,\varepsilon}](t))\frac{d\rho_{1,\varepsilon}}{dt}\\
&+\partial_{xC}F(X_\varepsilon(t)+x_{m,0,\varepsilon},\rho_{1,\varepsilon},\bar{f}[q_{1,\varepsilon}](t))\frac{d\bar{f}[q_{1,\varepsilon}](t)}{dt}=0.
\end{align*}

We may control $|\dot{X}_\varepsilon(t)|$ using the above expression and the following facts: the coefficient of $\dot{X}_\varepsilon$ is bounded away from $0$ as a consequence (A2); the results of Step 3  control  $\frac{d}{dt}(\bar{f}[q_{1,\varepsilon}](t))$; derivatives of $F$ are bounded in terms the model parameters; we have shown that $\left\vert\frac{d\rho_{1,\varepsilon}}{dt}\right\vert<K_4\varepsilon^{-1}$.  Altogether this shows 
\[|\dot{X}_\varepsilon|<K_5\varepsilon^{-1},~~t\in[0,T_\varepsilon]\] where $K_5$ depends only on $K_4,L_1$, and the model parameters except for $\varepsilon$. This suffices to shows the bound on $\dot{X}_\varepsilon$ required for (H7) for $t\in[0,T_\varepsilon]$.

\textbf{Step 5.}
To show  (H8) we can use the general result \cref{lma:KPPBound} to bound $\left\vert\frac{d\rho_{1,\varepsilon}}{dt}\right\vert\leq{}C_1+C_2\varepsilon^{-1}e^{-\frac{\gamma}{\varepsilon^2}t}$, where those constants are independent of $\varepsilon$.  We check that for $\nu_\varepsilon=\left\vert\frac{d\rho_{1,\varepsilon}}{dt}\right\vert$, for $\varepsilon$ sufficiently small, we have
\begin{align*}
\int_{0}^{t}|\nu_\varepsilon|(s)ds&\leq{}\int_{0}^{t}(C_1+C_2\varepsilon^{-1}e^{-\frac{\gamma}{\varepsilon^2}t})e^{-\frac{\eta}{\varepsilon^2}(t-s)}ds,\\
    &\leq{}\frac{C_2}{\eta}\varepsilon.
\end{align*}
This is sufficient to verify (H9) since it implies the lag $Y_\varepsilon=O(\varepsilon)$.

\textbf{Step 6.} Finally, we  move from local to global bounds. To do this, it's sufficient to note that all auxiliary results rely only the moment bounds \eqref{eqn:M2kcloseIP}-\eqref{eqn:M1closeIP} and the initial conditions. \cref{thm:main} then implies that $T_\varepsilon=\infty$. 

The remaining part of (H7), the convergence of $\tilde{m}_\varepsilon(x,t)$ and $X_\varepsilon(t)$ to $\widetilde{M}_0(x,t)$ and $X_0(t)$, now follows from $\rho_{1,\varepsilon}(t)=I(\bar{Z}_\varepsilon)+O(\varepsilon^{1-\delta})$ (itself a consequence of consequence \cref{lma:KPPBound}), and the estimates $\bar{\delta}[q_{1,\varepsilon}]=\delta(\bar{Z}_{\varepsilon})+O(\varepsilon^{1-\delta})$, and $\bar{f}[q_{1,\varepsilon}]=f(\bar{Z}_{\varepsilon})+O(\varepsilon^{1-\delta})$ obtained at the end of Step 1.
\vspace{-\belowdisplayskip}\[\]

\end{proof}

\section{Numerics}
\label{sec: numerics}
\begin{figure*}[t]
 \centering
  \subfloat[]{%
      \includegraphics[width=0.45\linewidth]{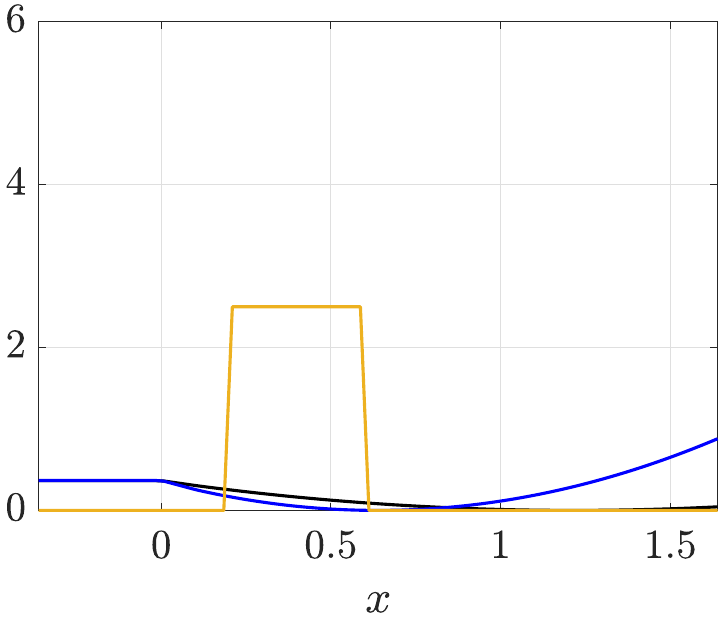}%
            \label{ScaledDistribution1:a}%
        }\hfill
        \subfloat[]{%
            \includegraphics[width=0.45\linewidth]{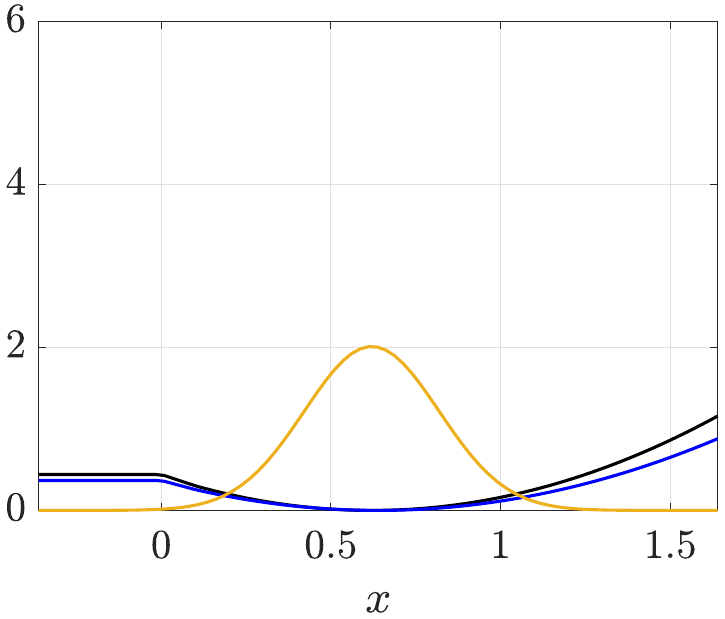}%
            \label{ScaledDistribution1:b}%
        }\\
         \subfloat[]{%
      \includegraphics[width=0.45\linewidth]{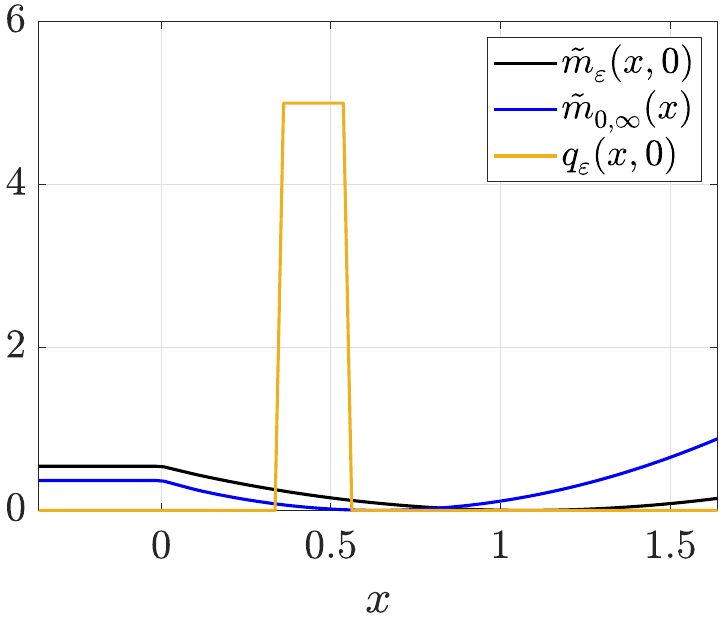}%
            \label{ScaledDistribution2:a}%
        }\hfill
        \subfloat[]{%
            \includegraphics[width=0.45\linewidth]{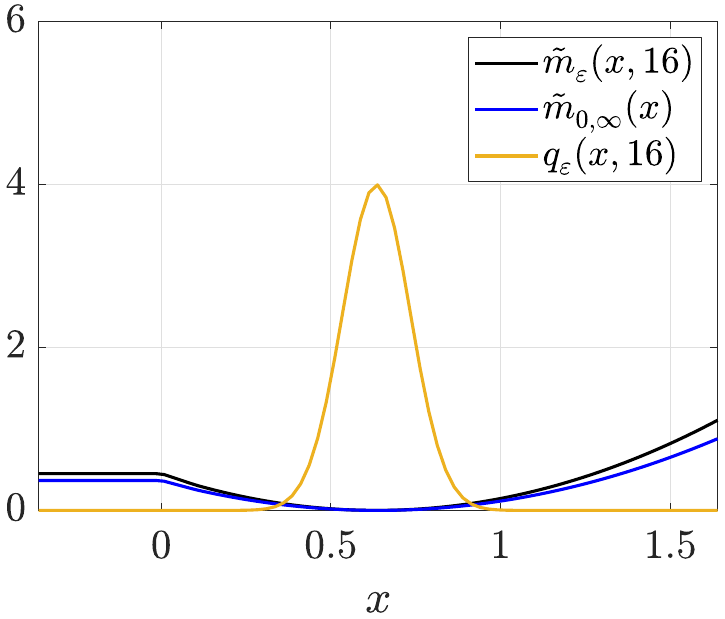}%
            \label{ScaledDistribution2:b}%
        }\\
        \caption{Phenotypic distribution $q_{\varepsilon}(x,t)$ (yellow) and the  mortality $\widetilde{M}(x,t)$ (black) for $\varepsilon=0.1$ (a),(b) $\varepsilon=0.2$ (c),(d), and $t=16$ (a),(c) $t=2$ (b),(d). Also shown is the equilibrium distribution $m_{0,\infty}(x)$ (black) to which large time solution of $m_{\varepsilon}(x,t)$ converges as $\varepsilon\rightarrow{0}$}
        \label{fig:ScaledDistributionMortality}
\end{figure*}
\begin{figure*}[t]
 \centering
  \subfloat[]{%
            \includegraphics[width=0.45\linewidth]{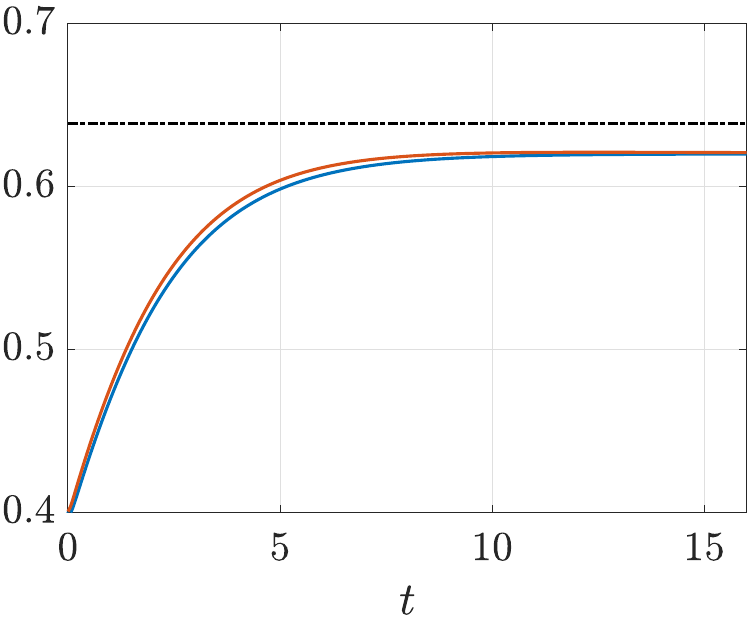}%
            \label{subfig:a}%
        }\hfill
        \subfloat[]{%
            \includegraphics[width=0.45\linewidth]{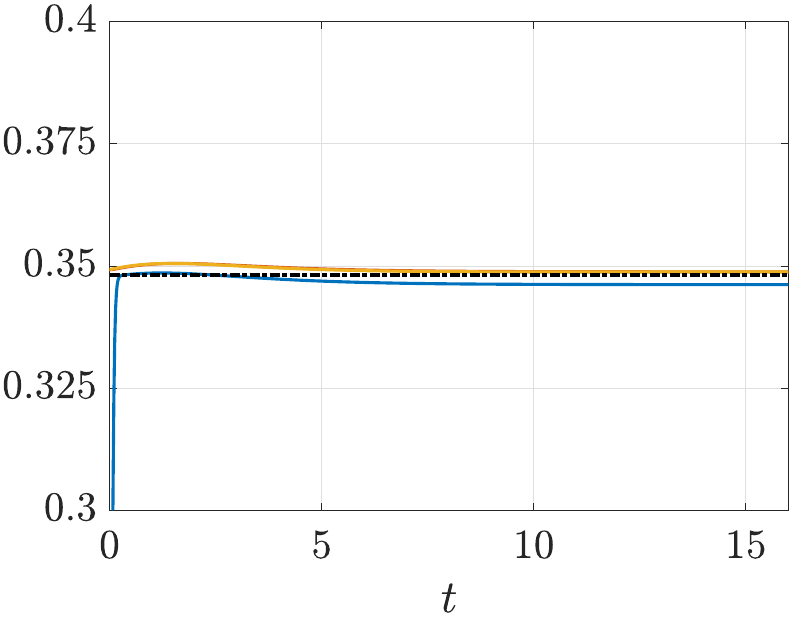}%
            \label{subfig:b}%
        }\\
        \subfloat[]{%
            \includegraphics[width=0.45\linewidth]{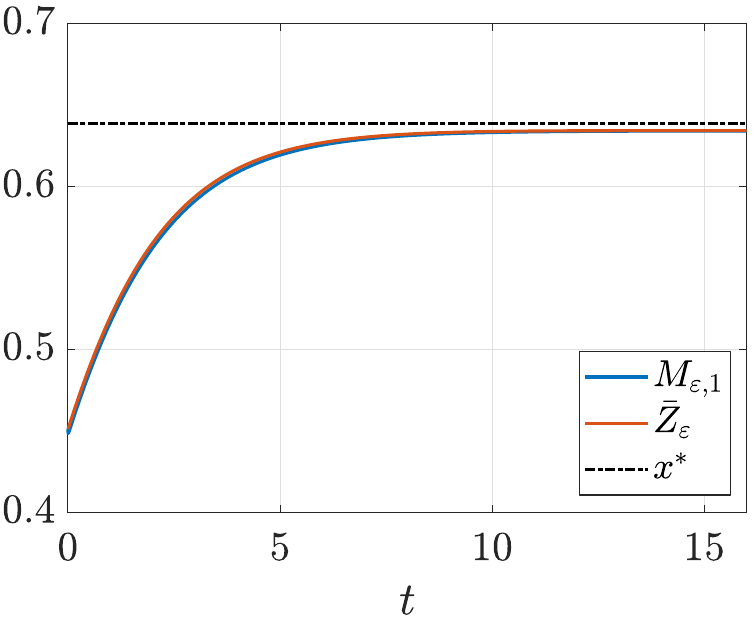}%
            \label{subfig:c}%
        }\hfill
        \subfloat[]{%
            \includegraphics[width=0.45\linewidth]{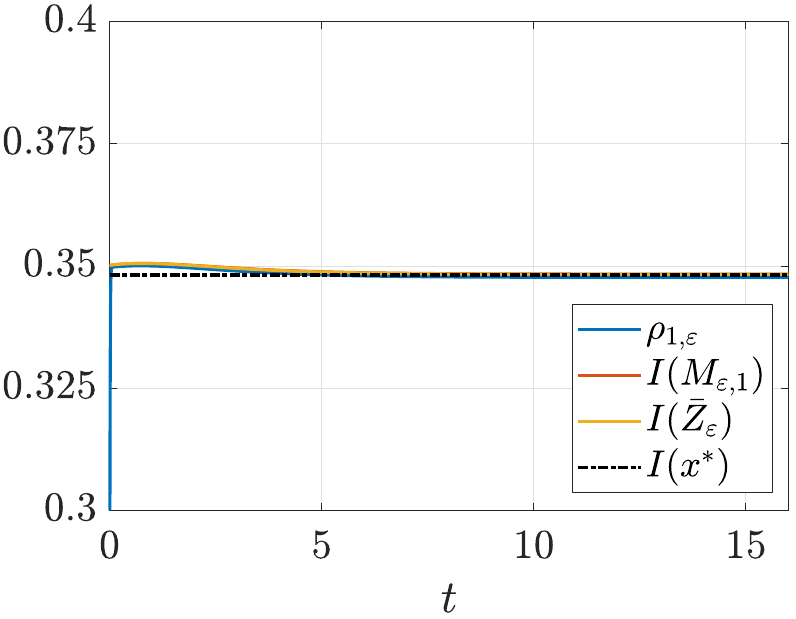}%
            \label{subfig:d}%
        }
        \caption{Trajectories of the mean trait $M_{\varepsilon,1}(t)$ and total population size for (a)-(b) $\varepsilon=0.2$ and (c)-(d) $\varepsilon=0.1$. For (a) and (c) the mean trait (solid blue) is compared compared against the approximate mean trait $\bar{Z}_\varepsilon(t)$ (red), and the fixed point $x^{*}$ (black dashed). For (c) and (d) $\varepsilon=0.1$. For (b) and (d) the mean trait (solid blue) is compared to $I(M_{\varepsilon,1})$ (solid red), $I(\bar{Z}_\varepsilon)$ (solid yellow), and the (proposed) limiting steady state $I(x^{*})$ (black dashed).}
        \label{fig:MeanTraitTotalPop}
\end{figure*}

To illuminate our theorems, we numerically solve system \eqref{eqn:IP_standard_form} for the specific choices $f(x)=\max(x,0.1)^2$, $\delta(x)=\frac{2}{5}\min(\kappa,\max(x,0))$, $r_1=1$, $M_{\varepsilon,1}(0)=0.8+\varepsilon/2$ (where $\varepsilon=0.2$ or $0.1$), which are in accordance with assumptions (A0)-(A5), except for the fact they are not globally $C^3(\mathbb{R})$. They are however locally $C^{3}(\mathbb{R})$ in the region of interest, and the results above are readily adapted to this case. We use a forward-in-time finite difference scheme on uniform discretisation of $[-L,L]$ into uniformly spaced grid points of width $\delta{x}=0.02$, with time-step $\delta{t}=\frac{1}{2}\varepsilon^2$. We solve with Dirichlet boundary conditions, and find that choosing $L=2$ is sufficiently large to mitigate errors introduced by using a finite domain.

We introduce here the following notation for the discrete interval $[[a,b]]=\{a,a+1,...,b\}$. Let  $x_i=x_0+i\delta{x}$ for $i\in[[-100,100]]$ and $t_n=n\delta{}t$ for $n\in[[0,N_t]]$ where $N_t=\lceil\frac{4}{\varepsilon^2}\rceil$. The initial condition is given by 
\[q_i^0=\frac{1}{2\varepsilon}\mathbf{1}_{\{x: |x-M_0|<\varepsilon\}}(x_i),\]
where $\mathbf{1}_{A}(x)$ is the indicator function of the set $A$, and we always specify $q^n$ is zero at the boundary points. For $i\in[[-100,100]]$ and $n\in[[1,N_t]]$ the numerical scheme is given as
\[\begin{cases}
q_i^{n+1}=q_i^n+(r_1{}(T^{n}_i-q_i^n)-(m_i^n-\bar{m}^n)q_i^n)\delta{t}\\
m_{i}^n=\rho^n\bar{f}^nf(x_i)-\delta(x_i),
\end{cases}\]
where averaged variables, denoted with a bar, are computed approximated using MATLABs \texttt{trapz} function using the function discrete approximation $q_i^n$. To improve stability, and since the operation should preserve probability distributions if taken over the entire domain, at each time step $q_i^n$ is normalized.

Figure \ref{fig:ScaledDistributionMortality} confirms that the mortality function $\tilde{m}_\varepsilon(x,t)$ converges, as $t\rightarrow\infty$, towards a limiting function $\widetilde{M}_{\varepsilon,\infty}
(x)$ which is independent of $t$. The phenotypic distribution closely approximates a Gaussian distribution given by Theorem \ref{thm:IP_main}.

Figure \ref{fig:MeanTraitTotalPop} shows the trajectory main trait and total population size. The trajectories for $M_{\varepsilon,1}(t)$ and $\bar{Z}_\varepsilon(t)$ are increasingly close for smaller $\varepsilon$, which in turn implies the same for $I(M_{\varepsilon,1})$ and $I(\bar{Z}_\varepsilon)$. As predicted, $\rho_{1,\varepsilon}(t)$ converges rapidly to its steady state which is at $O(\varepsilon^{1-\delta})$ distance from $K(\bar{Z}_\varepsilon(t))$.

\section{Discussion}\label{sec:Conclusion}
The first contribution of this work is the generalisation of the results in \cite{GHM} to the case of time-dependent birth and mortality rates. Our second contribution is to apply generalised results to a predator-prey model incorporating phenotype-dependent interactions between species. In the course of the asymptotic analysis, we develop assumptions on biologically relevant features such as the trait-dependent contact rate between the species, finding the reduced model is appropriate when, roughly speaking, the initial condition is concentrated not too far away from a local minimum of the mortality function. The leading order equations for the population size and mean trait obtained reveal how the ecological interaction between the two species shapes the trait-dependent mortality of the prey species. In this way, we show the moment-based approach is applicable to a model incorporating eco-evolutionary dynamics, and is able to capture key information about the transient and long-term dynamics, such as the mean and variance of the population. We note the analysis could be extended to higher order moments too. 

The analysis presented here could be applied  to model (\ref{eqn:evolving_prey}) in the case of $\tau>0$, or even for a trait-structure to the predator population as well. A natural question is: how do eco-evolutionary dynamics change as the time-scales between the two populations diverge in either direction? Applying the theory in this case would require showing that the mortality of each population satisfies (H0)-(H9).

Lastly, we highlight that fluctuating ecological dynamics may significantly alter evolutionary outcomes (see \cite{best2023fluctuating} for a review). Therefore an interesting direction for future work is to consider systems where $X_\varepsilon(t)$ does not approach a constant for large $t$. This is relevant when the eco-evolutionary dynamics possess dynamic attractors, such as limit cycles, rather than fixed points as in our predator-prey example. Our assumption (H7) that the optimal trait $X_\varepsilon(t)$ satisfies $\int_{0}^{t}|\dot{X}_\varepsilon(s)|\exp\left(-\frac{\eta}{\varepsilon^2}\right)\leq{}\varepsilon{}L_X$ for some constant $L_X$  allows the total variation in $X_\varepsilon(t)$ to be $O(\varepsilon^{-1})$, so Proposition \ref{prop:InfiniteT_eps}, for the generalized model is already applicable to systems where fluctuations in $X_\varepsilon(t)$ are not too large. We expect similar results are possible even in the case where the optimal trait varies yet more quickly and is oscillatory (as would be relevant for systems that are attracted to limit cycles). In this case, it should be possible to weaken the condition since the sign of $\dot{X}_\varepsilon(t)$ alternates, leading to cancellations.

\section*{Acknowledgments}
M.H.D is supported by EPSRC (grant EP/Y008561/1) and a Royal International Exchange Grant IES-R3-223047. F.S. acknowledges funding for a UKRI Future Leaders Fellowship MR/T043571/1. BvR is grateful for the support of the University of Birmingham.

This preprint has not undergone any post-submission improvements or corrections. The Version of Record of this article is
published in the Journal of Dynamics and Differential Equations, and is available online at https://doi.org/10.1007/s10884-026-10482-6.

\printbibliography
\end{document}